\documentclass[9pt]{extarticle} 
\usepackage[ 
  paperwidth=5.88in,paperheight=8.60in,
  textwidth=4.3in,textheight=6.7in]{geometry} 

\usepackage{fancyhdr} 
\usepackage{amsmath, amsthm, amsfonts, amssymb, mathdots}
\usepackage{stmaryrd}
\usepackage{paralist}
\usepackage{mathtools}
\usepackage{tikz}
\usepackage{graphicx}
\usepackage{cleveref}
\newtheorem{theorem}{Theorem}

\newtheorem{lemma}{Lemma}
\newtheorem{corollary}{Corollary}

\theoremstyle{definition}
\newtheorem{remark}{Remark}
\newtheorem{example}{Example}
\fancypagestyle{headings}{%
\fancyhf{} 
\fancyfoot[L]{
}

}

\newcommand{\RR}{\mathbb{R}}
\newcommand{\CC}{\mathbb{C}}

\newcommand{\MM}{\textnormal{\textbf{M}}}

\makeatletter 
\renewcommand{\maketitle}{
\pagestyle{plain}
\vspace*{\baselineskip}
\begin{center}
\MakeUppercase{\small Philip Saltenberger} \\ \emph{Institute for Numerical Analysis, TU Braunschweig} \\ \emph{Braunschweig, Germany} \\ (E-Mail: philip.saltenberger@tu-bs.de)

\vspace*{3\baselineskip}

\MakeUppercase{\large \@title}
\end{center} \vskip-\baselineskip
}
\makeatother

\usepackage{titlesec} 
\titleformat{\section}[hang]{\Large\scshape}{\thesection. }{0pt}{\centering}[]
\titleformat{\subsection}[hang]{\large\scshape}{\thesubsection. }{0pt}{\centering}[]
\titleformat{\subsubsection}[hang]{\scshape}{\thesubsubsection. }{0pt}{\centering}[]

\title{Structure-preserving eigenvalue modification of symplectic matrices and matrix pencils}

\begin{document}
\maketitle
\begin{abstract}
A famous theorem by R. Brauer shows how to modify a single eigenvalue of a matrix $A$ by a rank-one update without changing the remaining eigenvalues. A generalization of this theorem (due to R. Rado) is used to change a pair of eigenvalues $\lambda, 1/\lambda$ of a symplectic matrix $S$ in a structure-preserving way to desired target values $\mu, 1/\mu$. Universal bounds on the relative distance between $S$ and the newly constructed symplectic matrix $\hat{S}$ with modified spectrum are given. The eigenvalues Segre characteristics of $\hat{S}$ are related to those of $S$ and a statement on the eigenvalue condition numbers of $\hat{S}$ is derived. The main results are extended to matrix pencils.
\end{abstract}

\section{Introduction} \label{sec:intro}

In numerical linear algebra and matrix analysis one occasionally encounters the necessity of modifying special eigenvalues of a matrix without altering its remaining eigenvalues. Techniques for changing certain eigenvalues of a matrix have, for instance, been applied to solve nonnegative inverse eigenvalue problems \cite{Perfect1955, Soto2006} or, in form of deflation methods, to remove dominant eigenvalues in eigenvalue computations \cite[Sec.\,4.2]{saad_evals}. Furthermore, the task of modifying eigenvalues of matrices is of interest in stability and feedback of linear systems \cite[\S\,25]{delchamps}, \cite[Sec.\,2.3]{Bru2012} or for passivity and eigenvalue assignment in control design \cite{alex}. One basic result on how a single eigenvalue of a matrix may be changed without modifying any other eigenvalues is due to R. Brauer and can be found in \cite[Sec.\,1]{Bru2012}, \cite{Soto2006}.

\begin{theorem}[Brauer] \label{thm:brauer}
Let $A \in \MM_{n}(\CC)$ have eigenvalues $\lambda_1, \ldots , \lambda_n \in \CC$ and let $x_1 \in \CC^n$ be an eigenvector for $\lambda_1$. Then, for any $c \in \CC^n$, the matrix $\hat{A} = A + x_1c^T \in \MM_n(\CC)$ has the eigenvalues $\lambda_1 + c^Tx_1, \lambda_2, \ldots , \lambda_n$.
\end{theorem}

This work is concerned with the purposive change of certain eigenvalues of matrices with symplectic structure. A complex $2n \times 2n$ matrix $S \in \MM_{2n}(\CC)$ is called \textit{symplectic}, if\footnote{Here and in the following, $^T$ denotes the transpose of a (maybe complex) matrix or vector, not its conjugate transpose.}
\begin{equation} S^TJ_{2n}S = J_{2n} =: J, \quad \textnormal{where} \; J_{2n} = \begin{bmatrix} 0_{n \times n} & I_n \\ -I_n & 0_{n \times n} \end{bmatrix}. \label{equ:def_symplectic} \end{equation}
 Defining $S^\star := J^TS^TJ$, we see that \eqref{equ:def_symplectic} is equivalent to $S^\star S = I_{2n}$. Therefore, a symplectic matrix $S$ is always nonsingular and $S^\star = S^{-1}$. In consequence, as $S^\star$ is similar\footnote{By definition, $S^\star$ is similar to $S^T$ and by the Taussky-Zassenhaus Theorem \cite{taussky}, $S^T$ is similar to $S$.} to $S$, the eigenvalues of a symplectic matrices arise in pairs $\lambda_j, \lambda_j^{-1}$, $j=1, \ldots , n$, where $\lambda_j$ and $\lambda_j^{-1}$ have the same Segre characteristic. Recall that for an eigenvalue $\lambda$ of $S$, its Segre characteristic is the sequence of sizes of the Jordan blocks of $S$ with eigenvalue $\lambda$ in non-increasing order \cite{shapiro}. We denote the Segre characteristic of an eigenvalue by $((\cdot, \ldots , \cdot))$. It is now immediate that Theorem \ref{thm:brauer} can in general not be used for a structure-preserving, symplectic change of eigenvalues. In fact, for a structure-preserving eigenvalue modification, the change of $\lambda_j$ and $\lambda_j^{-1}$ must take place simultaneously.

Without any structure-preservation in mind, changing two (or more) eigenvalues simultaneously is possible with the following generalization of Theorem \ref{thm:brauer} attributed to R. Rado. It can be found in \cite{Perfect1955}, see also \cite[Sec.\,3]{Bru2012}.

\begin{theorem}[Rado] \label{thm:rado}
Let $A \in \MM_{n}(\CC)$ have eigenvalues $\lambda_1, \ldots , \lambda_n \in \CC$ and let $x_1, \ldots , x_k \in \CC^n$ be linearly independent eigenvectors for $\lambda_1, \ldots , \lambda_k$. Set $X = [ \; x_1 \; \cdots \; x_k \; ] \in \MM_{n \times k}(\CC)$. Then, for any matrix $C \in \MM_{n \times k}(\CC)$, the matrix $\hat{A} = A + XC^T$ has the eigenvalues $\mu_1, \ldots , \mu_k, \lambda_{k+1}, \ldots , \lambda_n$, where $\mu_j$, $j=1, \ldots , k$, are the eigenvalues of $\Omega = \textnormal{diag}(\lambda_1, \ldots , \lambda_n) + C^TX$.
\end{theorem}

In this work, we investigate how Theorem \ref{thm:rado} can be utilized to change a pair of eigenvalues $\lambda_j, \lambda_j^{-1}$ of a symplectic matrix $S \in \MM_{2n}(\CC)$ (or a symplectic matrix pencil) to desired target values $\mu, \mu^{-1}$ in a structure-preserving way without modifying any other eigenvalues of $S$. Considering Theorem \ref{thm:rado}, the starting point of our discussion is thus the following question:
\begin{itemize}
\item[] \textit{Let $S \in \MM_{2n}(\CC)$ be symplectic with eigenvalues $\lambda_1, \lambda_1^{-1}$,  $\ldots$ , $\lambda_n, \lambda_n^{-1}$, linearly independent eigenvectors $x_1, x_2 \in \CC^{2n}$ for $\lambda_1$ and $\lambda_1^{-1}$, respectively, and $X = [ \; x_1 \; x_2 \; ]$.  How has $C \in \MM_{2n \times 2}(\CC)$ to be chosen, such that $\hat{S} := S + XC^T$ is symplectic with eigenvalues $\mu, \mu^{-1}, \lambda_2, \lambda_2^{-1}, \ldots ,$ $\lambda_n, \lambda_n^{-1}$ for some given value $\mu \in \CC \setminus \lbrace 0 \rbrace$?}
\end{itemize}

The above-mentioned problem will be discussed in Section \ref{sec:symplectic_matrix}. In Section \ref{sec:min_norm} we investigate whether we can find an upper bound $b > 0$ that only depends on $\lambda_1, \mu, x_1$ and $x_2$ that assures the existence of a symplectic matrix $\hat{S} \in \MM_{2n}(\CC)$ with eigenvalues $\mu, \mu^{-1}, \lambda_2, \lambda_2^{-1}, \ldots , \lambda_n, \lambda_n^{-1}$ and relative distance $\Vert \hat{S} - S \Vert/ \Vert S \Vert \leq b$. We derive distinguished matrices $\hat{S}_1, \hat{S}_2$ for which such a bound $b$ can be neatly expressed and related to the relative change in the eigenvalue, i.e. $| \lambda_1 - \mu|/| \lambda_1|$.
We discuss commutativity relations between $S$ and $\hat{S}$ in Section \ref{sec:commutativity} and characterize the Segre characteristics of the eigenvalues of $\hat{S}$ in Section \ref{sec:eigenvectors}. The results of Section \ref{sec:commutativity} will come in handy here to find a condition on the simultaneous diagonalizability of $S$ and $\hat{S}$. In Section \ref{sec:pencils} we partially extend our results from Section \ref{sec:symplectic_matrix} to symplectic matrix pencils.

\subsection{Notation}
The set of all $m \times n$ matrices over $\mathbb{K}$ (where we use either $\mathbb{K} = \CC$ or $\mathbb{K} = \RR$) is denoted by $\MM_{m \times n}(\mathbb{K})$. Whenever $n=m$ we write $\MM_n(\mathbb{K})$ instead of $\MM_{n \times n}(\CC)$. For $J_{2n} \in \MM_{2n}(\RR)$, see \eqref{equ:def_symplectic}, we simply write $J$ and add the index whenever it is necessary to specify the size of $J$. The range of a matrix $A \in \MM_{m \times n}(\CC)$ is the vector space spanned by its columns and is denoted $\textnormal{range}(A)$.
For $A \in \MM_{m \times n}(\CC)$, we denote the Moore-Penrose pseudoinverse of $A$ by $A^+$. In case $m > n$ and $\textnormal{rank}(A)=n$, we have $A^+ = (A^HA)^{-1}A^H$ so that $A^+A = I_n$, while for $n > m$ and $\textnormal{rank}(A) = m$, $A^+ = A^H(AA^H)^{-1}$ yields $AA^+ = I_m$. The superscript $^H$ always denotes the conjugate transpose of a matrix or vector while $^T$ is used for the pure transposition.
Whenever $\lambda \in \CC$ is some complex number, we denote by $\mathfrak{R}(\lambda)$ and $\mathfrak{I}(\lambda)$ its real and imaginary part, respectively. Complex conjugation of a number $x = a + \imath b \in \CC$ is denoted by a bar, i.e. $\overline{x} = a - \imath b$.

\section{Symplectic Eigenvalue Modification}
\label{sec:symplectic_matrix}

Let $S \in \MM_{2n}(\CC)$ be a symplectic matrix (see \eqref{equ:def_symplectic}) with eigenvalues $\lambda_1, \lambda_1^{-1}$, $\ldots$, $\lambda_n , \lambda_n^{-1}$ and let $\mu \in \CC \setminus \lbrace 0 \rbrace$ be given. Furthermore, assume $x_1, x_2 \in \CC^{2n}$ are linearly independent\footnote{If $\lambda_1 \neq \lambda_1^{-1}$, then $x_1$ and $x_2$ are necessarily linear independent. Therefore, the linear independence is only a restrictive requirement if $\lambda_1 = \lambda_1^{-1}$, i.e. $\lambda_1 = \pm 1$.} eigenvectors of $S$ for $\lambda_1$ and $\lambda_1^{-1}$, respectively, and define $X = [ \; x_1 \; x_2 \; ] \in \MM_{2n \times 2}(\CC)$.
In this section, our goal is to determine all possible matrices $C \in \MM_{2n \times 2}(\CC)$ such that $\hat{S} := S + XC^T$ is again symplectic and has the eigenvalues $\mu, \mu^{-1}, \lambda_2, \lambda_2^{-1}, \ldots , \lambda_n, \lambda_n^{-1}$. To this end, we will make use of Rado's theorem and derive a structure-preserving version of Theorem \ref{thm:rado} (see Theorem \ref{thm:symp_change1}).

As it will become clear later, it seems appropriate to consider the situations $x_1^TJx_2 \neq 0$ and $x_1^TJx_2 =0$ seperately. First, we assume that for the eigenvalues $\lambda_1, \lambda_1^{-1}$ there exist eigenvectors $x_1,x_2 \in \CC^{2n}$ such that $x_1^TJx_2 \neq 0$ (this immediately implies $x_1$ and $x_2$ to be linearly independent). In this case, we can assume w.\,l.\,o.\,g. $x_1^TJx_2 = 1$, which can be achieved by a scaling of $x_1$ and/or $x_2$. That is, we have $X^TJX = J_2$. For the matrix $\hat{S} := S+XC^T$ to be symplectic, it has to hold that
\begin{equation} \hat{S}^TJ\hat{S} = \big( S + XC^T \big)^TJ \big( S+XC^T \big) = J. \label{equ:symp_condition} \end{equation}
Using $S^TJS = J$, \eqref{equ:symp_condition} is equivalent to the matrix equation
\begin{equation} CX^TJS + S^TJXC^T + CX^TJXC^T = 0 \label{equ:symp_matrixequation1} \end{equation}
for the unknown matrix $C \in \MM_{2n \times 2}(\CC)$. Notice that \eqref{equ:symp_matrixequation1} can be rewritten as
\begin{equation}
C (X^TJS + J_2C^T ) = - S^TJXC^T
\label{equ:symp_matrixequation1a} \end{equation}
using $X^TJX = J_2$. Since $S^TJX \in \MM_{2n \times 2}(\CC)$ is a matrix of full rank, \eqref{equ:symp_matrixequation1a} immediately implies $\textnormal{range}(C) \subseteq \textnormal{range}(S^TJX)$ for any solution $C$. Thus, for every $C$ satisfying \eqref{equ:symp_matrixequation1}, there is a matrix $R = [r_{ij}]_{ij} \in \MM_2(\CC)$ such that $C = S^TJXR^T$. Plugging this ansatz into \eqref{equ:symp_matrixequation1}, we obtain a $2n \times 2n$ equation for $R$, namely
$$ S^TJXR^TX^TJS + S^TJXRX^TJ^TS + S^TJXR^TJ_2RX^TJ^TS = 0.$$
Replacing $X^TJS$ by $-X^TJ^TS$ this can be rewritten as
\begin{equation} S^TJX \big( R - R^T + R^T J_2 R \big) X^TJ^TS = 0. \label{equ:symp_matrixequation1b} \end{equation}
Finally, we may multiply \eqref{equ:symp_matrixequation1b} with the pseudo inverses $(S^TJX)^+$ from the left and with $(X^TJ^TS)^+$ from the right to obtain
\begin{equation} R - R^T + R^TJ_2R = 0, \label{equ:symp_matrixequation2} \end{equation}
which is a matrix equation for $R$ of size $2 \times 2$ that is equivalent to \eqref{equ:symp_matrixequation1b}.
As $R-R^T$ and $R^TJ_2R$ are both skew-symmetric, their diagonals are identically zero. Comparing the entries of $R-R^T$ and $R^TJ_2R$ in the (1,2) position, we obtain the condition
\begin{equation} r_{12}-r_{21} + r_{11}r_{22}-r_{12}r_{21}=0\label{equ:symp_matrixequation3} \end{equation}
for \eqref{equ:symp_matrixequation2} to hold (comparing the elements in the $(2,1)$ position certainly gives the same condition with a minus sign). In summary, \textit{a matrix of the form $\hat{S} = S + XC^T$ is symplectic if and only if $C = S^TJXR^T$ for some matrix $R = [r_{ij}]_{ij} \in \MM_2(\CC)$ whose entries satisfy \eqref{equ:symp_matrixequation3}.}

Next, to achieve the desired eigenvalue modification, according to Theorem \ref{thm:rado} we need to assure that the eigenvalues of
\begin{equation} \Omega := \Lambda + C^TX = \begin{bmatrix} \lambda_1 & 0 \\ 0 & \lambda_1^{-1} \end{bmatrix} + C^TX = \begin{bmatrix} \lambda_1 & 0 \\ 0 & \lambda_1^{-1} \end{bmatrix} + RX^TJ^TSX \label{equ:omega} \end{equation}
become equal to $\mu$ and $\mu^{-1}$. To this end, recall that $SX = X \textnormal{diag}(\lambda_1, \lambda_1^{-1})$ (by construction of $X$). Thus $\textnormal{diag}(\lambda_1, \lambda_1^{-1}) + RX^TJ^TSX = \textnormal{diag}(\lambda_1, \lambda_1^{-1}) - RX^TJX \textnormal{diag}(\lambda_1, \lambda_1^{-1})$ which, since $X^TJX = J_2$, yields
\begin{equation} \Omega = \begin{bmatrix} \lambda_1 & 0 \\ 0 & \lambda_1^{-1} \end{bmatrix} - RJ_2 \begin{bmatrix} \lambda_1 & 0 \\ 0 & \lambda_1^{-1} \end{bmatrix} = \begin{bmatrix} \lambda_1 + \lambda_1 r_{12} & -\lambda_1^{-1}r_{11} \\ \lambda_1 r_{22} & \lambda_1^{-1} - \lambda_1^{-1}r_{21} \end{bmatrix}. \label{equ:omega1} \end{equation}
The characteristic polynomial of $\Omega$ is
$$ p(z) = z^2 - \big( \lambda_1(1+r_{12}) + \lambda_1^{-1}(1-r_{21}) \big)z + (1+r_{12})(1-r_{21}) + r_{11}r_{22},$$
which should, by Theorem \ref{thm:rado}, be equal to $q(z) = (z - \mu)(z - \mu^{-1}) = z^2 - (\mu + \mu^{-1}) +1$ to achieve that $\hat{S}$ will have the eigenvalues $\mu$ and $\mu^{-1}$. This gives two more conditions: one the one hand $\lambda_1(1+r_{12}) + \lambda_1^{-1}(1-r_{21}) = \mu + \mu^{-1}$, i.e.
\begin{equation} \lambda_1 r_{12} - \lambda_1^{-1}r_{21} = (\mu + \mu^{-1}) - (\lambda_1 + \lambda_1^{-1}). \label{equ:symp_matrixequation4} \end{equation}
One the other hand, $(1+r_{12})(1-r_{21}) + r_{11}r_{22} = 1$. The latter condition, however, is equal to condition \eqref{equ:symp_matrixequation3} obtained for the symplectic structure above. Thus, \textit{additionally to \eqref{equ:symp_matrixequation3}, which is required for $S+XC^T$ to be symplectic, the equation \eqref{equ:symp_matrixequation4} has to hold to achieve that $\mu, \mu^{-1}$ become eigenvalues of $S + XC^T$.}
In conclusion, we obtain the following version of Theorem \ref{thm:rado} that answers the question stated in Section \ref{sec:intro} on the eigenvalue modification for symplectic matrices.

\begin{theorem} \label{thm:symp_change1}
Let $S \in \MM_{2n}(\CC)$ be symplectic with eigenvalues $\lambda_1, \lambda_1^{-1}, \lambda_2, \lambda_2^{-1},$ $\ldots , \lambda_n, \lambda_n^{-1}$ and let $\mu \in \CC \setminus \lbrace 0 \rbrace$ be given. Let $x_1, x_2 \in \CC^{2n}$ be eigenvectors for $\lambda_1$ and $\lambda_1^{-1}$, respectively, normalized such that $X^TJX = J_2$ for $X = [ \; x_1 \; x_2 \; ] \in \MM_{2n \times 2}(\CC)$ and set $d := (\mu + \mu^{-1} ) - ( \lambda_1 + \lambda_1^{-1})$. Then the matrix
\begin{equation} \hat{S} := S + X C^T \in \MM_{2n}(\CC) \label{equ:symp_change1} \end{equation}
is symplectic and has the eigenvalues $\mu, \mu^{-1}, \lambda_2, \lambda_2^{-1}, \ldots , \lambda_n, \lambda_n^{-1}$ if and only if $C^T = R X^TJ^TS$ for some matrix $R=[r_{ij}]_{ij} \in \MM_2(\CC)$ whose entries satisfy the conditions
\begin{align}
d &= \lambda_1r_{12} - \lambda_1^{-1} r_{21}, \; \textnormal{and} \label{equ:symp_change1a} \\
0 &= r_{12} - r_{21} + r_{11} r_{22} - r_{12}r_{21}.  \label{equ:symp_change1b}
\end{align}
\end{theorem}

Notice that the matrix $\hat{S} = S + XRX^TJ^TS$ in \eqref{equ:symp_change1} can also be expressed as $\hat{S} = (I_{2n} + XRX^TJ^T)S$ or as
\begin{equation} \begin{aligned} \hat{S} = S + XR \Lambda^{-1}X^TJ^T &= S + XR \begin{bmatrix} \lambda_1^{-1} & 0 \\ 0 & \lambda_1 \end{bmatrix} X^TJ^T
\end{aligned} \label{equ:hat_S} \end{equation}
according to the relation $\Lambda^{-1}X^TJ^T = X^TJ^TS$ (where $\Lambda = \textnormal{diag}(\lambda_1, \lambda_1^{-1})$). Furthermore, we see from \eqref{equ:symp_change1a} and \eqref{equ:symp_change1b} that there exist infinitely many possible choices for $R$ that realize the desired eigenvalue modification.

Next, we discuss the case that the eigenvectors $x_1,x_2 \in \CC^{2n}$ of the symplectic matrix $S$ for $\lambda_1$ and $\lambda_1^{-1}$, respectively, satisfy $x_1^TJx_2 = 0$ and how this condition effects the result from Theorem \ref{thm:symp_change1}. To this end, first notice that a symplectic matrix $S \in \MM_{2n}(\CC)$ need in fact not have eigenvectors $x_1, x_2$ for $\lambda_1$ and $\lambda_1^{-1}$ that satisfy $x_1^TJx_2 \neq 0.$ A situation of this kind arises for the symplectic matrix
$$S = \begin{bmatrix} \lambda_1 & 1 & 0 & 0 \\ 0 & \lambda_1 & 0 & 0 \\ 0 & 0 & \lambda_1^{-1} & 0 \\ 0 & 0 & -\lambda_1^{-2} & \lambda_1^{-1} \end{bmatrix}$$
and its eigenvalue $\lambda_1$. The only eigenvectors for $\lambda_1$ and $\lambda_1^{-1}$ are $e_1$ and $e_4$, respectively, and we have $e_1^TJ_4e_4 = 0$. Thus, Theorem \ref{thm:symp_change1} cannot be applied. A simple sufficient (but not necessary) criterion to assure that eigenvectors with $x_1^TJx_2 \neq 0$ must exist, is that $S$ is a diagonalizable matrix, cf. \cite[Lem.\,3, Cor.\,3.1]{laub} and Corollary \ref{cor:diag} below.

Whenever $x_1,x_2 \in \CC^{2n}$ are eigenvectors of $S$ for $\lambda_1$ and $\lambda_1^{-1}$ with $x_1^TJx_2 = 0$, then $X^TJX = 0$ follows for $X = [ \; x_1 \; x_2 \;]$. In this case, it follows from \eqref{equ:symp_matrixequation1} that \eqref{equ:symp_matrixequation1a} takes the form
$$ C X^TJS = - S^TJXC^T.$$
Again we obtain $\textnormal{range}(C) \subseteq \textnormal{range}(S^TJX)$, so there has to exist some matrix $R = [r_{ij}]_{ij} \in \MM_2(\CC)$ such that $C = S^TJXR^T$. However, despite the concrete form of $R$,  analogously to \eqref{equ:omega} we obtain
$$ \Omega = \begin{bmatrix} \lambda & 0 \\ 0 & \lambda^{-1} \end{bmatrix} + C^TX = \begin{bmatrix} \lambda & 0 \\ 0 & \lambda^{-1} \end{bmatrix} - RX^TJSX  = \begin{bmatrix} \lambda & 0 \\ 0 & \lambda^{-1} \end{bmatrix} $$
since $SX = X \textnormal{diag}(\lambda, \lambda^{-1})$ and $X^TJX = 0$. Thus, \textit{even if $R$ is chosen according to \eqref{equ:symp_change1b} such that $\hat{S} = S + XRX^TJ^TS$ is symplectic, no change in the eigenvalues can be achieved.} In consequence, a change of an eigenvalue pair $\lambda_1, \lambda_1^{-1}$ of a symplectic matrix by Rado's theorem in a structure-preserving way is only possible if there exist eigenvectors $x_1$ and $x_2$ for $\lambda_1$ and $\lambda_1^{-1}$, respectively, such that $x_1^TJx_2 \neq 0$. In the next section, we derive a universal criterion on the existence of such eigenvectors.

\subsection{Applying Theorem \ref{thm:symp_change1}: a criterion}

We will now characterize those symplectic matrices $S \in \MM_{2n}(\CC)$, for which an eigenvalue adjustment according to Theorem \ref{thm:symp_change1} is possible. The condition derived below involves the Segre characteristic of the eigenvalue $\lambda_1 \in \sigma(S)$ to be modified.

First, let $\lambda_1 \in \sigma(S)$ and $Sx_1 = \lambda_1 x_1$ and $Sx_2 = \lambda_1^{-1}x_2$. Now suppose at least one of both vectors, e.g. $x_1$, belongs to a nontrivial\footnote{By nontrivial, we mean a Jordan chain of length $\geq 2$ while a trivial Jordan chain refers to a chain of length one.} Jordan chain, that is, there is some $z \in \CC^{2n}$ such that $(S - \lambda_1 I_{2n})z = x_1$ (and possibly more generalized eigenvectors beside $z$). Then we have
\begin{equation}
\begin{aligned}
x_1^TJx_2 = \big( (S - \lambda_1 I_{2n})z \big)^TJx_2 &= z^TS^TJx_2 - \lambda_1 z^TJx_2 \\ &= z^TJJ^TS^TJx_2 - \lambda_1 z^TJx_2 \\ &= z^TJS^{-1}x_2 - \lambda_1 z^TJx_2 \\ &= \lambda_1 z^TJx_2 - \lambda_1 z^TJx_2 = 0
\end{aligned}
\label{equ:zero_product} \end{equation}
as $J^TS^TJ = S^\star = S^{-1}$ and $S^{-1}x_2 = \lambda_1 x_2$. In consequence, $x_1^TJx_2 = 0$ whenever $x_1,x_2$ are eigenvectors of $S$ for $\lambda_1$ and $\lambda_1^{-1}$, respectively, and at least one of them belongs to a nontrivial Jordan chain. \textit{In other words, we may have $x_1^TJx_2 \neq 0$ only in case both $x_1$ and $x_2$ belong to trivial Jordan chains.} Next, we show that in case $x_1$ belongs to a trivial Jordan chain there must exist $x_2$ (also from a trivial Jordan chain) such that $x_1^TJx_2 \neq 0$.

To this end, assume that $\lambda_1 \in \CC$ is an eigenvalue of the symplectic matrix $S \in \MM_{2n}(\CC)$ with $p \geq 1$ ones in its Segre characteristic (that is, there are $p$ Jordan blocks of size $1 \times 1$, i.e. $p$ trivial Jordan chains, and possibly other Jordan blocks of size $\geq 2$). Then there exists a matrix $F \in \MM_{2n}(\CC)$ transforming $S$ to the following Jordan form
\begin{equation} F^{-1}SF =: G = \begin{bmatrix} \begin{array}{c|c} \begin{array}{ccc} \lambda_1 & & \\ & \ddots & \\ & & \lambda_1 \end{array} & 0 \\ \hline  0 & \hat{G} \end{array} \end{bmatrix} \label{equ:special_jordan} \end{equation}
where the upper-left block is $\lambda_1I_p$ and $\hat{G}$ contains all other Jordan blocks (note that there might also be other Jordan blocks for $\lambda_1$ of size $\geq 2$ contained in $\hat{G}$).
Now define
\begin{equation} x_1 := Fe_1 \quad \textnormal{and} \quad \tilde{f}^H := e_1^TF^{-1}. \label{equ:leftright0} \end{equation}
Then $x_1$ is a right eigenvector of $S$ for $\lambda_1$ ($Sx_1 = \lambda_1 x_1$) and $\tilde{f}$ is a left eigenvector of $S$ for $\lambda_1$ ($\tilde{f}^HS = \lambda_1 \tilde{f}^H$). Certainly, $\tilde{f}^Hx_1 = 1$. Now we define $x_2^T := \tilde{f}^HJ$. Then we have
\begin{equation} \begin{aligned}
\tilde{f}^HS = \lambda_1 \tilde{f}^H \; &\Leftrightarrow \; x_2^TJ^TS = \lambda_1 x_2^TJ^T \\ &\Leftrightarrow \; x_2^TJ^TSJ = \lambda_1x_2^T \\ &\Leftrightarrow \; x_2^TS^{-T} = \lambda_1x_2^T \; \Leftrightarrow \; Sx_2 = \lambda_1^{-1} x_2.
\end{aligned} \label{equ:rightleft} \end{equation}
It follows that $x_2$ is an eigenvector of $S$ for $\lambda_1^{-1}$. Now we obtain
$$ x_1^TJx_2 = x_2^TJ^Tx_1 = \tilde{f}^Hx_1 = 1 \neq 0.$$
In conclusion, for any eigenvector $x_1$ of $S$ for $\lambda_1$ belonging to a trivial Jordan chain, there always exists an eigenvector $x_2$ of $S$ for $\lambda_1^{-1}$ such that $x_1^TJx_2 \neq 0.$ Recall from our observation \eqref{equ:zero_product} above, that $x_2$ must also be a vector from a trivial Jordan chain. We conclude our findings in the following theorem.

\begin{theorem} \label{thm:nec_condition}
Let $S \in \MM_{2n}(\CC)$ be symplectic with $\lambda_1 \in \sigma(S)$ and let $\mu \in \CC \setminus \lbrace 0 \rbrace$ be given. Then Theorem \ref{thm:symp_change1} is applicable to $S$ for $\lambda_1$ and $\mu$, i.e. there exist eigenvectors $x_1, x_2 \in \CC^{2n}$ for $\lambda_1$ and $\lambda_1^{-1}$, respectively, with $x_1^TJx_2 \neq 0$, if and only if the Segre characteristic of $S$ for $\lambda_1$ contains a one, that is, it has the form $((\star, \star, \cdots, \star, 1))$. In particular, eigenvectors $x_1$ for $\lambda_1$ and $x_2$ for $\lambda_1^{-1}$ with $x_1^TJx_2 \neq 0$ always belong to trivial Jordan chains of $S$.
\end{theorem}

Do not overlook that Theorem \ref{thm:nec_condition} applies also for $\lambda_1 = \lambda_1^{-1}$, i.e. $\lambda_1 = \pm 1$. In this case $\lambda_1 \in \sigma(S)$ necessarily has an even multiplicity and an even number of Jordan blocks of the same size, so a Segre characteristic of the form $((\star, \star, \cdots, \star, 1))$ implies that there appears at least another one, i.e. $((\star, \star, \cdots, \star, 1,1))$. Then the reasoning in \eqref{equ:special_jordan}, \eqref{equ:leftright0} and \eqref{equ:rightleft} applies in the same way.
If $S$ is diagonalizable, the Segre characteristic of $S$ for any eigenvalue $\lambda_j \in \sigma(S)$ consists only of ones. So we immediately obtain the following corollary.

\begin{corollary} \label{cor:diag}
Let $S \in \MM_{2n}(\CC)$ be symplectic with $\lambda_1 \in \sigma(S)$ and let $\mu \in \CC \setminus \lbrace 0 \rbrace$ be given. Then Theorem \ref{thm:symp_change1} is applicable to $S$ for $\lambda_1$ and $\mu$ if $S$ is diagonalizable.
\end{corollary}

\section{Bounding the relative change}
\label{sec:min_norm}
Let $S \in \MM_{2n}(\CC)$ be symplectic. For the matrix $\hat{S} = S + XRX^TJ^TS$ in \eqref{equ:symp_change1} we immediately obtain a bound on its (absolute or relative) change in norm with respect to $S$. That is,
\begin{equation} \Vert S - \hat{S} \Vert \leq \Vert R \Vert \Vert X \Vert \Vert X^T \Vert \Vert S \Vert \quad \textnormal{and} \quad \frac{\Vert S - \hat{S} \Vert}{\Vert S \Vert} \leq  \Vert R \Vert \Vert X \Vert \Vert X^T \Vert \label{equ:main_bound} \end{equation}
hold for any submultiplicative and unitarily invariant matrix norm $\Vert \cdot \Vert$. In this section, we intend to derive explicit bounds of the relative distance between $S$ and $\hat{S}$ for certain choices of $R = [r_{ij}]_{ij} \in \MM_2(\CC)$. To this end, we assume $\Vert \cdot \Vert = \Vert \cdot \Vert_F$ so that $\Vert X \Vert_F = \Vert X^T \Vert_F$ holds and the upper bound in \eqref{equ:main_bound} reduces to $\Vert R \Vert_F \Vert X \Vert_F^2$.

To bound the relative change $\Vert \hat{S} - S \Vert_F/\Vert S \Vert_F$ with respect to $S$ consider again \eqref{equ:symp_change1a} and \eqref{equ:symp_change1b}. The solution set to \eqref{equ:symp_change1a} is an affine subspace of $\CC^2$ and all solutions may be parameterized as
\begin{equation} r_{12} = \eta \lambda_1^{-1}, \qquad r_{21} = - \lambda_1d + \eta \lambda_1, \qquad \eta \in \CC. \label{equ:r12r21} \end{equation}
Plugging these expressions for $r_{12}$ and $r_{21}$ into \eqref{equ:symp_change1b} yields a polynomial in $\eta$, i.e.
\begin{equation} p(\eta) = - \eta^2 + \eta \left( \lambda_1^{-1} + d - \lambda_1 \right) + d\lambda_1 + r_{11}r_{22}. \label{equ:eta_pol} \end{equation}
Thus, depending on $r_{11}$ and $r_{22}$ (which can both be arbitrary), in \eqref{equ:eta_pol} there are always two solutions $\eta_1, \eta_2 \in \CC$ of $p(\eta) = 0$
and, in consequence, two matrices
\begin{equation} R(\eta_j, r_{11},r_{22}) := \begin{bmatrix} r_{11} & \eta_j \lambda_1^{-1} \\ \lambda_1(\eta_j-d) & r_{22} \end{bmatrix}, \quad j=1,2, \label{equ:general_R} \end{equation}
so that their entries satisfy \eqref{equ:symp_change1a} and \eqref{equ:symp_change1b}.

To find some $R \in \MM_2(\CC)$ that yields a small norm $\Vert R \Vert_F$ and thus a small bound in \eqref{equ:main_bound}, it seems natural to consider the case $r_{11}r_{22}=0$, in particular $r_{11}=r_{22}=0$\footnote{Certainly, choosing $r_{11}=0$ and $r_{22} \neq 0$ gives the same roots of $p(\eta)=0$ in \eqref{equ:eta_pol}, and thus the same values for $r_{12}$ and $r_{21}$, but a larger Frobenius norm of $R$ than choosing $r_{11}=r_{22}=0$.}.
The two possible roots of $p(\eta)$ for $r_{11}r_{22}=0$ are $\eta_1 = \mu - \lambda_1$ and $\eta_2 = \mu^{-1} - \lambda_1$. The matrices $R_1 := R(\eta_1,0,0)$ and $R_2 := R(\eta_2,0,0)$ that arise according to \eqref{equ:general_R}
 are thus given by
\begin{equation} R_1 =  \begin{bmatrix} 0 & \lambda_1^{-1}(\mu - \lambda_1) \\ \mu^{-1}(\mu - \lambda_1)& 0 \end{bmatrix}, \quad R_2 =  \begin{bmatrix} 0 & \mu^{-1}(\lambda_1^{-1}- \mu) \\ \lambda_1(\lambda_1^{-1} - \mu) & 0 \end{bmatrix}. \label{equ:min_R} \end{equation}
Using $R_1$ and $R_2$ in \eqref{equ:min_R}, explicit bounds can be found on $\Vert \hat{S} - S \Vert_F /\Vert S \Vert_F$. According to \eqref{equ:main_bound} and \eqref{equ:min_R} such a bound $b \geq 0$ only depends on $\lambda_1$, $\mu$ and the eigenvectors of $S$ for $\lambda_1$ and $\lambda_1^{-1}$ and guarantees the existence of a symplectic matrix $\hat{S} \in \MM_{2n}(\CC)$ that solves the problem from Section \ref{sec:intro} with $\Vert \hat{S} - S \Vert_F /\Vert S \Vert_F \leq b$. To formulate these bounds, we impose a condition on $X$ to estimate $\Vert X \Vert_F$ without computing the norm. In particular, we assume the eigenvectors $x_1,x_2 \in \CC^{2n}$ of $S$ for $\lambda_1$ and $\lambda_1^{-1}$, respectively, to be normalized, i.e. $\Vert x_1 \Vert_2 = \Vert x_2 \Vert_2 = 1$ and $X \in \MM_{2n \times 2}(\CC)$ to be of the form
\begin{equation} X = \frac{1}{\sqrt{x_1^TJx_2}} \begin{bmatrix} x_1 & x_2 \end{bmatrix}. \label{equ:X_normalized} \end{equation}
Then $X^TJX = J_2$ holds and  it follows that
\begin{equation} \Vert X \Vert_{F}^2 = \left| \frac{1}{\sqrt{x_1^TJx_2}} \right|^2 \left\Vert \begin{bmatrix} x_1 & x_2 \end{bmatrix} \right\Vert_{F}^2 = \frac{1}{| x_1^TJx_2|} \left\Vert \begin{bmatrix} x_1 & x_2 \end{bmatrix} \right\Vert_{F}^2 = \frac{2}{|x_1^TJx_2|} \label{equ:normX} \end{equation}
The value $1/|x_1^TJx_2|$ has a nice interpretation whenever $\lambda_1$ is a  \textit{simple} eigenvalue of $S$ and $\Vert x_1 \Vert_2 = \Vert x_2 \Vert_2 = 1$ holds. To see this, recall that, whenever $A \in \MM_n(\CC)$ has a simple eigenvalue $\lambda \in \CC$ (i.e. its algebraic multiplicity equals one), then
$$ \kappa(A, \lambda) := \frac{\Vert u \Vert_2 \Vert v \Vert_2}{|v^Hu|}$$
is called its condition number, where $u \in \CC^{2n}$ and $v \in \CC^{2n}$ are right and left eigenvectors of $A$ for $\lambda$ (i.e. $Au = \lambda u$ and $v^HA = \lambda v^H$).
It is a measure on how sensitive $\lambda$ reacts to small changes in the matrix $A$, see \cite[Sec.\,3.3]{saad_evals}. As we have seen in \eqref{equ:rightleft} above, $(x_2^TJ^T)S = \lambda_1 (x_2^TJ^T)$ whenever $x_2$ satisfies $Sx_2 = \lambda_1^{-1}x_2$. Thus, for simple $\lambda_1$ (which implies that $\lambda_1^{-1}$ is simple as well) we can choose $u = x_1$ and $v^H = x_2^TJ^T$ so that
$$ \kappa(S, \lambda_1) = \frac{\Vert u \Vert_2 \Vert v \Vert_2}{|v^Hu|} = \frac{\Vert x_1 \Vert_2 \Vert J\overline{x}_2 \Vert_2}{|x_2^TJ^Tx_1|} = \frac{1}{|x_1^TJx_2|}$$
since $\Vert x_1 \Vert_2 = 1$ and $\Vert J\overline{x}_2 \Vert_2 = \Vert \overline{x}_2 \Vert_2 = \Vert x_2 \Vert_2 = 1$. We can now formulate the following theorem which follows directly from the bound in \eqref{equ:main_bound}, the observation in \eqref{equ:normX} and the Frobenius norms of the matrices $R_1,R_2$ in \eqref{equ:min_R}.

\begin{theorem} \label{thm:coarse_bound}
Let $S \in \MM_{2n}(\CC)$ be symplectic with $\lambda_1, \lambda_1^{-1} \in \sigma(S)$ and let $\mu \in \CC \setminus \lbrace 0 \rbrace$ be given. Let $x_1, x_2 \in \CC^{2n}$ be normalized eigenvectors for $\lambda_1$ and $\lambda_1^{-1}$, respectively, and $X \in \MM_{2n \times 2}(\CC)$ as in \eqref{equ:X_normalized}. Define
$$ \Phi := \frac{2}{|x_1^TJx_2|} \; \big( = 2 \kappa(S, \lambda_1) \; \textnormal{if $\lambda_1$ is simple} \big).$$
\begin{enumerate}
\item[$(i)$] Let $\hat{S}_1 = S + XR_1X^TJ^TS$ be constructed according to Theorem \ref{thm:symp_change1} with $R_1$ from \eqref{equ:min_R}. Then
\begin{alignat}{3}
&\frac{\Vert \hat{S}_1 - S \Vert_F}{\Vert S \Vert_F} &&\leq  \frac{| \lambda_1 - \mu |}{| \lambda_1|} &&\left[ \Phi  \sqrt{1 + \frac{| \lambda_1|^2}{| \mu|^2}} \right].  \label{equ:frob_bound_S1a}
\end{alignat}
\item[$(ii)$] Let $\hat{S}_2 = S + XR_2X^TJ^TS$ be constructed according to Theorem \ref{thm:symp_change1} with $R_2$ from \eqref{equ:min_R}. Then
\begin{alignat}{3}
&\frac{\Vert \hat{S}_2 - S \Vert_F}{\Vert S \Vert_F} &&\leq \frac{| \lambda_1^{-1}-\mu|}{| \lambda_1^{-1}|} &&\left[ \Phi \sqrt{1 + \frac{|\lambda_1^{-1}|^2}{| \mu|^2}}  \right]. \label{equ:frob_bound_S2a}
\end{alignat}
\end{enumerate}
\end{theorem}

As the following example shows, similar easy bounds can be found with the use of $R_1$ and $R_2$ when $\Vert \cdot \Vert_2$ is considered. In fact, in the 2-norm, such a bound can be sharp.

\begin{example}
For a symplectic matrix $S \in \MM_{2n}(\CC)$, the bound \eqref{equ:main_bound} for $\hat{S}_1 = S + XR_1X^TJ^TS$ with respect to $\Vert \cdot \Vert_2$ can easily be determined as
\begin{alignat}{2}
&\frac{\Vert \hat{S}_1 - S \Vert_2}{\Vert S \Vert_2} &&\leq \frac{| \lambda_1 - \mu |}{|\lambda_1|} \cdot \max \left\lbrace 1 , \frac{| \lambda_1|}{|\mu|} \right\rbrace \Vert X \Vert_2^2,  \label{equ:sharp_bound}
\end{alignat}
It can  be seen for $S = \textnormal{diag}(\Lambda, \Lambda^{-1})$ with $\Lambda = \textnormal{diag}(\lambda_1, \ldots , \lambda_n)$ that the bound in \eqref{equ:sharp_bound} can be sharp. In particular, with eigenvectors $e_1, e_{n+1} \in \RR^{2n}$ for $\lambda_1, \lambda_1^{-1}$, respectively, and $X = [ \; e_1 \; e_{n+1} \; ]$ we have 
$$ XR_1X^TJ^TS = \textnormal{diag} \big( r_{12} \lambda_1, 0, \ldots , 0, -r_{21} \lambda_1^{-1}, 0, \ldots , 0 \big)$$
with nonzero entries in the first and $(n+1)$st position. As $r_{12} \lambda_1 = \mu - \lambda_1$ and $-r_{21} \lambda_1^{-1} = \mu^{-1} - \lambda_1^{-1}$ we obtain under the assumption $| \lambda_1 - \mu | \geq | \lambda_1^{-1} - \mu^{-1} |$
$$ \Vert \hat{S} - S \Vert_2 = \Vert XR_1X^TJ^TS \Vert_2 = | \lambda_1 - \mu |$$
and so $\Vert \hat{S} - S \Vert_2/\Vert S \Vert_2 = | \lambda_1 - \mu |/| \lambda_1|$ if  $\lambda_1$ is the largest eigenvalue of $S$ in absolute value (i.e. $\Vert S \Vert_2 = |\lambda_1|$). On the other hand, for $X$ we certainly have $\Vert X \Vert_2^2 = 1$ and thus, whenever $| \mu | \geq | \lambda_1 |$, the bound on the right hand side in \eqref{equ:sharp_bound} also reduces to $| \lambda_1 - \mu |/| \lambda_1 |$.
\end{example}

\subsection{Improved distance bounds}
Although the bound in \eqref{equ:frob_bound_S1a}  nicely relates $\Vert \hat{S}_1 - S\Vert_F/\Vert S \Vert_F$ to the relative value change $| \lambda_1 - \mu|/| \lambda_1|$ and the condition number $\kappa(S, \lambda_1)$, it can be quite bad\footnote{The same is true for the bound in \eqref{equ:frob_bound_S2a}.}, see e.g. Fig. \ref{fig:upper} in Section \ref{sec:experiments}. In this section we derive sharper bounds under the additional assumption that $\Vert S \Vert_F$ is known.

As before, let $S \in \MM_{2n}(\CC)$ be symplectic with eigenvectors $x_1,x_2 \in \CC^{2n}$ for $\lambda_1, \lambda_1^{-1} \in \sigma(S)$, respectively, such that $x_1^TJx_2 \neq 0$ (and set $X = [ \, x_1 \; x_2 \, ]$).
As seen in \eqref{equ:hat_S}, 
we have for
$\Lambda := \textnormal{diag}(\lambda_1, \lambda_1^{-1})$ and $R \in \MM_2(\CC)$ that satisfies \eqref{equ:symp_change1a} and \eqref{equ:symp_change1b}
$$ \hat{S} = S + XRX^TJ^TS = S + XR \Lambda^{-1}X^TJ^T$$ and therefore $ \Vert S - \hat{S} \Vert_F = \Vert XR \Lambda^{-1}X^TJ^T \Vert_F = \Vert XR \Lambda^{-1}X^T \Vert_F$.
Whenever $R = R_j$ ($j=1,2$) from \eqref{equ:min_R}, then
$$ XR_j \Lambda^{-1}X^T = X \begin{bmatrix} 0 & \eta_j \\ \eta_j - d & 0 \end{bmatrix} X^T =: X \tilde{R}_j X^T, \quad \tilde{R}_j = R_j \Lambda^{-1},$$
according to \eqref{equ:r12r21}. Recall the solutions of $p(\eta)=0$ in \eqref{equ:eta_pol}, i.e. $\eta_1 = \mu - \lambda_1$ (corresponding to $R_1$) and $\eta_2 = \mu^{-1} - \lambda_1$ (corresponding to $R_2$). Instead of estimating $\Vert X \tilde{R}_j X^T \Vert_F$ by $\Vert \tilde{R}_j \Vert_F \Vert X \Vert_F^2$ we now intend to estimate $\Vert X \tilde{R}_j X^T \Vert_F$ directly.
To this end, assume that $X \in \MM_{2n \times 2}(\CC)$ has the form \eqref{equ:X_normalized} and $\hat{S}_j = S + XR_jX^TJ^TS$, $j=1,2$. Then we have for $Y = \sqrt{x_1^TJx_2} X = [\, x_1 \; x_2 \, ]$
$$ \begin{aligned} \Vert S - \hat{S}_j \Vert_F^2 &= \Vert X \tilde{R}_j X^T \Vert_F^2 = \frac{\textnormal{tr} \big( Y \tilde{R}_jY^T(Y \tilde{R}_jY^T)^H \big)}{|x_1^TJx_2|^2} \\ &= \frac{\textnormal{tr} \big( Y \tilde{R}_jY^T \overline{Y} \tilde{R}_j^HY^H \big)}{|x_1^TJx_2|^2}
= \frac{\textnormal{tr} \big( Y^HY \tilde{R}_j \overline{(Y^HY)} \tilde{R}_j^H \big)}{|x_1^TJx_2|^2}.
\end{aligned}$$
Now we further obtain
\begin{align}  | x_1^TJx_2|^2 &\Vert \hat{S}_j - S \Vert_F^2 = \textnormal{tr} \big( Y^HY \tilde{R}_j \overline{(Y^HY)} \tilde{R}_j^H \big) \notag \\ \quad &= \textnormal{tr} \left( \begin{bmatrix} 1 & x_1^Hx_2 \\ x_2^Hx_1 & 1 \end{bmatrix} \begin{bmatrix} 0 & \eta_j \notag \\ \eta_j - d & 0 \end{bmatrix} \begin{bmatrix} 1 & x_2^Hx_1 \\ x_1^Hx_2 & 1 \end{bmatrix} \begin{bmatrix} 0 & \overline{\eta}_j - \overline{d} \\ \overline{\eta}_j & 0 \end{bmatrix} \right) \notag \\
&= |x_1^Hx_2|^2 \overline{\eta}_j (\eta_j-d) + | \eta_j|^2 + |\eta_j-d |^2 + |x_1^Hx_2|^2 \eta_j ( \overline{\eta}_j - \overline{d}) \notag \\
& = |\eta_j|^2 + |\eta_j - d|^2 + 2 |x_1^Hx_2|^2 \cdot \mathfrak{R}(\eta_j(\overline{\eta}_j- \overline{d})). \label{equ:estimate_bound}
\end{align}

Note that $|x_1^Hx_2| \leq \Vert x_1 \Vert_2 \Vert x_2 \Vert_2 = 1$ as $x_1$ and $x_2$ are normalized. Furthermore, $\mathfrak{R}(\eta_j(\overline{\eta}_j- \overline{d})) = \mathfrak{R}(|\eta_j|^2 - \eta_j \overline{d}) = | \eta_j|^2 - \mathfrak{R}(\eta_j \overline{d})$, and we may now derive upper (and lower) bounds for \eqref{equ:estimate_bound} depending on whether this term is positive or negative.
\begin{enumerate}
\item[$(i)$] Suppose $\mathfrak{R}(\eta_j \overline{d}) < | \eta_j|^2$. Then $\mathfrak{R}(\eta_j(\overline{\eta}_j- \overline{d})) > 0$ follows and, since $|x_1^Hx_2|^2 \leq 1$, we can estimate from \eqref{equ:estimate_bound}, setting $|x_1^Hx_2|^2 = 1$,
$$ \begin{aligned} |x_1^TJx_2|^2 \Vert S - \hat{S}_j \Vert_F^2 &\leq |\eta_j|^2 + |\eta_j - d|^2 + 2 \cdot \mathfrak{R}(\eta_j(\overline{\eta}_j- \overline{d})) \\ &= \big( \eta_j + (\eta_j-d) \big) \big( \overline{\eta}_j + (\overline{\eta}_j-\overline{d} ) \big) = | 2 \eta_j - d |^2.\end{aligned}$$
On the other hand, changing the sign of $\mathfrak{R}(\eta_j(\overline{\eta}_j- \overline{d}))$  we certainly have
$$ \begin{aligned} |x_1^TJx_2|^2 \Vert S - \hat{S}_j \Vert_F^2 &\geq | \eta_j|^2 + |\eta_j - d|^2 - 2 \cdot \mathfrak{R}(\eta_j(\overline{\eta}_j- \overline{d})) \\
&= (\eta_j - (\eta_j - d))( \overline{\eta}_j - (\overline{\eta}_j - \overline{d})) = | d|^2.
\end{aligned} $$
\item[$(ii)$] Suppose $\mathfrak{R}(\eta_j \overline{d}) \geq | \eta_j|^2$. Then $\mathfrak{R}(\eta_j(\overline{\eta}_j- \overline{d})) \leq 0$ follows and we can estimate from \eqref{equ:estimate_bound}, setting again $|x_1^Hx_2|^2 = 1$,
$$  |x_1^TJx_2|^2 \Vert S - \hat{S}_j \Vert_F^2 \geq | \eta_j|^2 + | \eta_j-d|^2 + 2 \cdot \mathfrak{R}(\eta_j(\overline{\eta}_j- \overline{d})) = |2\eta_j-d|^2
 $$
while on the other hand, with a change of sign, we obtain
$$ \begin{aligned}
|x_1^TJx_2|^2 \Vert S - \hat{S}_j \Vert_F^2 &\leq | \eta_j|^2 + | \eta_j-d|^2 - 2 \cdot \mathfrak{R}(\eta_j(\overline{\eta}_j- \overline{d})) = |d|^2.
\end{aligned} $$

\end{enumerate}
Before we state our findings in the next theorem, notice that there are neat expressions for the terms $2\eta_j-d$, $j=1,2$, arising above, i.e.
$$ 2\eta_1 - d = \frac{\mu - \lambda_1}{\lambda_1} \big( \lambda_1 + \mu^{-1} \big), \quad 2\eta_2 - d = \frac{\lambda_1^{-1} - \mu}{\lambda_1^{-1}} \big(\mu^{-1} +  \lambda_1^{-1}\big).$$
As it turns out, also $d$ can be rewritten in a similar fashion as
\begin{equation} d = \frac{\mu - \lambda_1}{\lambda_1} \big( \lambda_1 - \mu^{-1} \big) = \frac{\lambda_1^{-1} - \mu}{\lambda_1^{-1}} \big( \mu^{-1} - \lambda_1^{-1} \big). \label{equ:expr_d} \end{equation}
Finally, the two conditions to be checked in $(i)$ and $(ii)$ above can be simplified. 
For $\eta_1 = \mu - \lambda_1$ one finds, after some reformulations,
$$ \begin{aligned}
\mathfrak{R} \big( \eta_1 \overline{d} \big) - |\eta_1|^2  &= \frac{1}{2} \big( (\mu - \lambda_1)\overline{d} + (\overline{\mu} - \overline{\lambda}_1)d \big) - | \mu - \lambda_1|^2 \\
&= \frac{1}{2} \left( \frac{|\mu - \lambda_1|^2 \overline{d}}{\overline{\mu} - \overline{\lambda}_1} + \frac{|\mu - \lambda_1|^2d}{\mu - \lambda_1} \right) - |\mu - \lambda_1|^2 \\
&= - | \mu - \lambda_1|^2 \left( 1 - \frac{1}{2} \left( \frac{\overline{d}}{\overline{\mu} - \overline{\lambda}_1} + \frac{d}{\mu - \lambda_1} \right) \right) \\ &= - |\mu - \lambda_1|^2 \left( 1 - \mathfrak{R} \left( \frac{d}{\mu - \lambda_1} \right) \right) \\
&= - |\mu - \lambda_1|^2 \left( 1 - \mathfrak{R} \left( \frac{\lambda_1 - \mu^{-1}}{\lambda_1} \right) \right) \\
&= -| \mu - \lambda_1|^2 \mathfrak{R}((\mu \lambda_1)^{-1})
\end{aligned} $$
where we used the first expression for $d$ in \eqref{equ:expr_d} in the second-last equation.

Thus $\mathfrak{R}(\eta_1 \overline{d}) \geq | \eta_j|^2$ holds if and only if $- | \mu - \lambda_1|^2 \mathfrak{R}((\mu \lambda_1)^{-1}) \geq 0$, which is the case if and only if $\mathfrak{R}(\lambda_1 \mu) \leq 0$ as $(\lambda_1 \mu)^{-1}$ and $\lambda_1 \mu$ are located in the same half plane. For $\eta_2 = \mu^{-1} - \lambda_1$ we obtain analogously
$$ \mathfrak{R} \big( \eta_2 \overline{d} \big) - |\eta_2|^2 = - | \lambda_1 \mu -1 |^2 \mathfrak{R} \big( (\overline{\lambda} \mu)^{-1} \big)$$
and so $\mathfrak{R}(\eta_2 \overline{d}) \geq | \eta_j|^2$ holds if and only if $\mathfrak{R}((\overline{\lambda} \mu)^{-1}) \leq 0$. This, in turn, holds if and only if $\mathfrak{R}(\overline{\lambda}_1 \mu) \leq 0$. In conclusion, we have proven the following theorem.

\begin{theorem} \label{thm:sharp_bound}
Let $S \in \MM_{2n}(\CC)$ be symplectic with $\lambda_1, \lambda_1^{-1} \in \sigma(S)$ and let $\mu \in \CC \setminus \lbrace 0 \rbrace$ be given. Let $x_1, x_2 \in \CC^{2n}$ be normalized eigenvectors for $\lambda_1$ and $\lambda_1^{-1}$, respectively, with $x_1^TJx_2 \neq 0$ and $X \in \MM_{2n \times 2}(\CC)$ as in \eqref{equ:X_normalized}.
Define $$ \Phi := \frac{1}{|x_1^TJx_2| \cdot \Vert S \Vert_F} \; \left( = \frac{\kappa(S, \lambda_1)}{\Vert S \Vert_F} \; \textnormal{if $\lambda_1$ is simple} \right)$$
\begin{enumerate}
\item[$(i)$] Let $\hat{S}_1 = S + XR_1X^TJ^TS$ be constructed according to Theorem \ref{thm:symp_change1} with $R_1$ from \eqref{equ:min_R}. Whenever $\mathfrak{R}(\lambda_1 \mu) \leq 0$, then
$$ \frac{|\lambda_1 - \mu|}{| \lambda_1|} \big( |\lambda_1 + \mu^{-1}| \Phi \big) \leq \frac{\Vert \hat{S}_1 - S \Vert_F}{\Vert S \Vert_F} \leq \frac{|\lambda_1 - \mu|}{| \lambda_1|} \big( |\lambda_1 - \mu^{-1}| \Phi \big).$$
If $\mathfrak{R}(\lambda_1 \mu) > 0$ the upper and lower bounds interchange.
\item[$(ii)$] Let $\hat{S}_2 = S + XR_2X^TJ^TS$ be constructed according to Theorem \ref{thm:symp_change1} with $R_2$ from \eqref{equ:min_R}. Whenever $\mathfrak{R}(\overline{\lambda}_1 \mu) \leq 0$, then
$$ \frac{|\lambda_1^{-1} - \mu|}{| \lambda_1^{-1}|} \big( |\lambda_1^{-1} + \mu^{-1}| \Phi \big) \leq \frac{\Vert \hat{S}_2 - S \Vert_F}{\Vert S \Vert_F} \leq \frac{|\lambda_1^{-1} - \mu|}{| \lambda_1^{-1}|} \big( |\lambda_1^{-1} - \mu^{-1}| \Phi \big).$$
If $\mathfrak{R}(\overline{\lambda_1} \mu) > 0$ the upper and lower bounds interchange.
\end{enumerate}
\end{theorem}

It is shown in Section \ref{sec:experiments} (see Fig. \ref{fig:upper}) that the bounds in Theorem \ref{thm:sharp_bound} are significantly sharper compared to the bounds in \eqref{equ:frob_bound_S1a} and \eqref{equ:frob_bound_S2a}.

\begin{remark} \label{rem:R1R2}
For any matrix $R = [r_{ij}] \in \MM_2(\CC)$ that satisfies the conditions \eqref{equ:symp_change1a} and \eqref{equ:symp_change1b} the bounds in \eqref{equ:main_bound} can easily be calculated.
However, there are several reasons for not considering other choices of $R$ (beside $R_1$ and $R_2$ from \eqref{equ:min_R}) in this section in detail:
\begin{enumerate}[(a)]
\item If $r_{11} \neq 0, r_{22} \neq 0$, there are two possibilities for $R$ whose entries $r_{12}$ and $r_{21}$ of $R$ depend on $c=r_{11}r_{22}$ through (one of) the zeros of $p(\eta)=0$, see \eqref{equ:eta_pol}. Thus, $r_{12}$ and $r_{21}$ involve the expression of a complex square root and there is no neat and compact expression for $\Vert R \Vert_F$ compared to \eqref{equ:frob_bound_S1a} and \eqref{equ:frob_bound_S2a} or to the formulas in Theorem \ref{thm:sharp_bound}. Furthermore, minimizing $\Vert S - \hat{S}\Vert_F$ with respect to the entries of $R$ under the side conditions \eqref{equ:symp_change1a} and \eqref{equ:symp_change1b} results in a difficult complex optimization problem for which the author is not aware of a closed form solution.
\item Among all matrices $R$ that satisfy the conditions \eqref{equ:symp_change1a} and \eqref{equ:symp_change1b} the matrices $R_1$ and $R_2$ from \eqref{equ:min_R} are the only possible choice when $\hat{S} = S+XRX^TJ^TS$ should inherit desirable properties (e.g. related to diagonalizability) from $S$. These distinguishing features of $R_1$ and $R_2$ are discussed in the upcoming sections.
\item All numerical experiments that have been performed indicate that rarely a matrix $R' \in \MM_2(\CC)$ different from $R_1$ and $R_2$ that satiesfies \eqref{equ:symp_change1a} and \eqref{equ:symp_change1b} was detected such that $\Vert \hat{S} - S \Vert_F/\Vert S \Vert_F$ for $\hat{S} = S+XR'X^TJ^TS$ was smaller than the minimum of $\Vert \hat{S}_1 - S \Vert/\Vert S \Vert_F$ and $\Vert \hat{S}_2 - S \Vert_F/\Vert S \Vert_F$. This is visualized in Section \ref{sec:experiments}, see Figure \ref{fig:surface}. 
\end{enumerate}

\end{remark}

\section{Segre characteristics and commutativity relations}

In this section we discuss how the Segre characteristics of eigenvalues are effected by a change of a symplectic matrix $S \in \MM_{2n}(\CC)$ to $\hat{S} \in \MM_{2n}(\CC)$ according to Theorem \ref{thm:symp_change1}. In particular, we will show that the Segre characteristics of the eigenvalues of $S$ and $\hat{S}$ are either the same or connected in a direct way. Furthermore, we make a statement on eigenvectors of $S$ and $\hat{S}$ that remain unchanged. Notice that, in the context of Theorem \ref{thm:rado}, the eigenvectors of  $A$ and $\hat{A} = A + XC^T$ are in general all different and \textit{not} related in an immediate fashion \cite{Bru2012} if no further restrictions are imposed on the form of $C$.  In this section we show that, in the structure-preserving context of Theorem \ref{thm:symp_change1}, the particular form of $C$ allows for some explicit statements. Furthermore, we derive statements on the diagonalizability of $\hat{S}$ and the simultaneous diagonalizability of $S$ and $\hat{S}$. To this end, we begin in Section \ref{sec:commutativity} with a result on the commutativity of $S$ and $\hat{S}$.

\subsection{The Commutativity of $S$ and $\hat{S}$}
\label{sec:commutativity}

Recall that the matrix $\hat{S}$ in \eqref{equ:symp_change1} can also be expressed as
\begin{equation} \hat{S} = \big( I_{2n} + XRX^TJ^T \big) S. \label{equ:mult_decomp} \end{equation}
Since $\hat{S}$ and $S$ are both symplectic, the matrix $\hat{S}S^{-1} = \hat{S}S^{\star} = I_{2n} + XRX^TJ^T$ is symplectic, too. 
As for any $A,B \in \MM_n(\CC)$ the matrices $AB$ and $BA$ always have the same eigenvalues \cite{johnson96}, beside $\hat{S}$, we may also define the symplectic matrix $\tilde{S} := S \big( I_{2n} + XRX^TJ^T \big) \in \MM_{2n}(\CC)$
that solves the eigenvalue modification problem stated in Section \ref{sec:intro}.
A question naturally arising is whether there is a connection between $\hat{S}$ from \eqref{equ:mult_decomp} and $\tilde{S}$. Such a connection is revealed in Theorem \ref{thm:commute} which shows a distinguishing feature of the matrices from \eqref{equ:min_R} among all matrices $R$ that satisfy \eqref{equ:symp_change1a} and \eqref{equ:symp_change1b}, see Remark \ref{rem:R1R2} (b). The result from Theorem \ref{thm:commute} will be used when the diagonalizability of $\hat{S}$ is investigated in Section \ref{sec:eigenvectors}.

\begin{theorem} \label{thm:commute}
Let $S \in \MM_{2n}(\CC)$ be symplectic with $\lambda_1 \in \sigma(S)$ and 
assume $\hat{S} = S + XRX^TJ^TS$ has been constructed according to Theorem \ref{thm:symp_change1}. 
Then the following is true:
\begin{enumerate}
\item In case $\lambda_1 \neq \pm 1$,
\begin{equation} \hat{S} = \big( I_{2n} + XRX^TJ^T \big)S = S \big( I_{2n} + XRX^TJ^T \big) = \tilde{S} \label{equ:comm_relation} \end{equation}
holds if and only if $R = [r_{ij}]_{ij} \in \MM_2(\CC)$ is one of the matrices in \eqref{equ:min_R}.
\item In case $\lambda_1 = \pm 1$, \eqref{equ:comm_relation} holds for any $R = [r_{ij}]_{ij} \in \MM_2(\CC)$ satisfying \eqref{equ:symp_change1a} and \eqref{equ:symp_change1b}.
\end{enumerate}
\end{theorem}

\begin{proof}
First, notice that $\hat{S} = \tilde{S}$ is equivalent to
\begin{equation} XRX^TJ^TS = S XRX^TJ^T. \label{equ:proof_comm0} \end{equation}
Multiplying both equations with $J$ (from the right) and using the relations $SX = X \Lambda$ (where $\Lambda = \textnormal{diag}(\lambda_1, \lambda_1^{-1})$) and $J^TSJ = S^{-T}$ yields $ XRX^TS^{-T} = X \Lambda RX^T.$ Moreover, $X^TS^{-T} = (S^{-1}X)^T = (X \Lambda^{-1})^T = \Lambda^{-1}X^T$ and, equivalently to \eqref{equ:proof_comm0}, it suffices to investigate the equation
\begin{equation} XR \Lambda^{-1}X^T = X \Lambda R X^T. \label{equ:proof_comm1} \end{equation}
Now, as $X \in \MM_{2n \times 2}(\CC)$ has full rank, \eqref{equ:proof_comm1} is (by the multiplication with $X^+$ from the left and $(X^T)^+$ from the right) equivalent to $R \Lambda^{-1} = \Lambda R$, that is, $\Lambda R \Lambda = R$. For $R = [r_{ij}]_{ij}$ we obtain
$$ \Lambda R \Lambda = \begin{bmatrix} \lambda_1 & 0 \\ 0 & \lambda_1^{-1} \end{bmatrix} \begin{bmatrix} r_{11} & r_{12} \\ r_{21} & r_{22} \end{bmatrix}  \begin{bmatrix} \lambda_1 & 0 \\ 0 & \lambda_1^{-1} \end{bmatrix} = \begin{bmatrix} \lambda_1^2r_{11} & r_{12} \\ r_{21} & \lambda_1^{-2}r_{22} \end{bmatrix}.$$
This shows that $\Lambda R \Lambda = R$ holds, in case $\lambda_1 \neq \pm 1$, if and only if $r_{11}=r_{22}=0$. The two possibilities for $R$ that satisfy the conditions \eqref{equ:symp_change1a} and \eqref{equ:symp_change1b} when $r_{11} = r_{22}=0$ are the matrices in \eqref{equ:min_R}. Furthermore, if $\lambda_1 = \pm 1$, the equation always holds. This completes the proof.
\end{proof}

Theorem \ref{thm:commute} shows that, in general (i.e. for $\lambda_1 \neq \pm 1$), only the two possible choices for $R$ in \eqref{equ:min_R} produce commutativity of $S$ and $I_{2n} + XRX^TJ^T$ (i.e. to have $\hat{S} = \tilde{S}$). In the special case $\lambda_1 = \pm 1$, any matrix $R$ determined from \eqref{equ:symp_change1a} and \eqref{equ:symp_change1b} will cause this commutativity relation.

\subsection{Segre characteristics}
\label{sec:eigenvectors}

Let $S \in \MM_{2n}(\CC)$ be symplectic with eigenvalues $\lambda_1, \lambda_1^{-1},$ $\ldots , \lambda_n, \lambda_n^{-1}$ and let $\mu \in \CC \setminus \lbrace 0 \rbrace$ be given. To analyse the consequences of the change $S \mapsto \hat{S} = S + XRX^TJ^TS$ on the Segre characteristics of the eigenvalues of $S$ and $\hat{S}$, we discuss the cases of $\lambda_1, \lambda_1^{-1}$ (the eigenvalues that are changed), $\mu, \mu^{-1}$ (the values $\lambda_1$ and $\lambda_1^{-1}$ are changed to) and all other eigenvalues (which are the same for $S$ and $\hat{S}$) separately.
As before, let $x_1, x_2 \in \CC^{2n}$ be eigenvectors for $\lambda_1$ and $\lambda_1^{-1}$, respectively, normalized such that $X^TJ_{2n}X = J_2$ for $X = [ \; x_1 \; x_2 \; ] \in \MM_{2n \times 2}(\CC)$ and assume $\hat{S} = S + XRX^TJ^TS$ has been constructed as in  Theorem \ref{thm:symp_change1}.

We first consider eigenvalues different from $\lambda_1, \lambda_1^{-1}, \mu$ and $\mu^{-1}$. These eigenvalues and their algebraic multiplicities are the same for $S$ and $\hat{S}$ and we show that their eigenspaces and Jordan chains (thus, in consequence, their Segre characteristics) remain completely unchanged. To prove this, we need the following fact about the matrix $S$ and its (generalized) eigenvectors (see also \cite[Sec.\,2]{laub}): assume that $\lambda$ is some eigenvalue of $S$ different from $\lambda_1$ and $\lambda_1^{-1}$ and let $y_1,y_2, \ldots , y_p \in \CC^{2n}$ ($p \geq 1$) be a Jordan chain for $S$ and $\lambda$, i.e. it holds that $(S - \lambda I_{2n})y_1 = 0$ and $(S - \lambda I_{2n})y_{k+1} = y_k$ for $k=1, \ldots , p-1$. Then $X^TJy_k = 0$ follows for any $k=1, \ldots , p.$ To see this, first consider the eigenvector $y_1$ of $S$ for $\lambda$. We have
$$\lambda_1 x_1^TJy_1 = x_1^TS^TJy_1 = x_1^TJJ^TS^TJy_1 = x_1^TJS^{-1}y_1 = \lambda^{-1} x_1^TJy_1$$
This shows that $x_1^TJy_1 = 0$ if $\lambda \neq \lambda_1^{-1}$. Similarly, $x_2^TJy_1 = 0$ follows for $\lambda \neq \lambda_1$. Now assume that $x_i^TJy_\ell = 0$ holds for $i=1,2$ and $\ell = 1, \ldots , k$. For $(S - \lambda I_{2n})y_{k+1} = y_k$ we thus obtain
$$ \begin{aligned}
0 = x_1^TJy_k = x_1^TJ(S - \lambda I_{2n})y_{k+1} &= x_1^TJSy_{k+1} - \lambda x_1^TJy_{k+1} \\
&= x_1^TS^{-T}Jy_{k+1} - \lambda x_1^TJy_{k+1} \\
&= \lambda_1^{-1} x_1^TJy_{k+1} - \lambda x_1^TJy_{k+1} \\
&= (\lambda_1^{-1} - \lambda) x_1^TJy_{k+1}.
\end{aligned} $$
Therefore, again $x_1^TJy_{k+1} = 0$ follows whenever $\lambda \neq \lambda_1^{-1}$. With the same reasoning we obtain $x_2^TJy_{k+1} = 0$ for $\lambda \neq \lambda_1$. In conclusion we have $X^TJy_k = 0$ for any $k=1, \ldots , p$ whenever $\lambda_1 \neq \lambda \neq \lambda_1^{-1}.$

\begin{lemma} \label{lem:eigenvectors_0}
Let $S \in \MM_{2n}(\CC)$ be symplectic with $\lambda_1 \in \sigma(S)$ and let $\mu \in \CC \setminus \lbrace 0 \rbrace$ be given. Suppose that $\hat{S} \in \MM_{2n}(\CC)$ has been constructed according to Theorem \ref{thm:symp_change1}. Then for any $\lambda \in \sigma(\hat{S})$ which is neither equal to $\lambda_1$ or $\lambda_1^{-1}$ nor equal to $\mu$ or $\mu^{-1}$ the Segre characteristics of $\lambda$ as an eigenvalue of $S$ and $\hat{S}$ and their corresponding Jordan chains, respectively, are identical.
\end{lemma}
\begin{proof}
Assume $\lambda \in \sigma(\hat{S})$ is neither equal to $\lambda_1$ or $\lambda_1^{-1}$ nor equal to $\mu$ or $\mu^{-1}$. By construction of $\hat{S}= S + XRX^TJ^TS$, $\lambda$ is an eigenvalue of both $S$ and $\hat{S}$ with the same algebraic multiplicities. Whenever $y_1 \in \CC^{2n}$ is an eigenvector of $S$ for $\lambda$ it is also an eigenvector of $\hat{S}$ for $\lambda$ since $X^TJy_1 = 0$, which implies
\begin{equation} \hat{S}y_1 = Sy_1 + XRX^TJ^TSy_1 = Sy_1 - \lambda XRX^TJy_1 = Sy_1 = \lambda y_1. \label{equ:single_eig} \end{equation}
Next, let $y_1, \ldots , y_p \in \CC^{2n}$ be a Jordan chain for $S$ and $\lambda$. Then 
$$ \begin{aligned}
\big( \hat{S} - \lambda I_{2n} \big) y_{i+1} &= \big( S + XRX^TJ^TS \big) y_{i+1} - \lambda y_{i+1} \\
&= \big( y_i + \lambda y_{i+1} \big) + XRX^TJ^TSy_{i+1} - \lambda y_{i+1} \\ &= y_i - XRX^TJ(y_i + \lambda y_{i+1}) = y_i
\end{aligned} $$
since $X^TJy_k=0$ for any $y_k$, $k=1, \ldots , p$, from the Jordan chain. Inductively, this shows that $y_1, \ldots , y_p$ remains to be a Jordan chain of $\hat{S}$ for $\lambda$. Therefore, the Segre characteristic for $\lambda$ of $S$ and $\hat{S}$ and the corresponding Jordan chains are the same.
\end{proof}

Next we consider $\lambda_1$ and $\lambda_1^{-1}$. When $S \in \MM_{2n}(\CC)$ is transformed to $\hat{S} = S + XRX^TJ^TS$ and (one instance of) $\lambda_1, \lambda_1^{-1}$ is replaced by $\mu$ and $\mu^{-1}$, the Segre characteristic of $\lambda_1$ for $\hat{S}$ is necessarily different from its Segre characteristic for $S$ due to the eigenvalue modification that has taken place (if $\lambda_1$ is a simple eigenvalue of $S$, then it is not even an eigenvalue of $\hat{S}$ anymore). However, if the algebraic multiplicity of $\lambda_1$ as an eigenvalue of $S$ is $\geq 2$, then the Segre characteristics of $\lambda_1$ as an eigenvalue of $S$ and $\hat{S}$ are connected in an easy fashion (see Theorem \ref{lem:eigval_lambda} below). This is obviously false in the general context of Rado's Theorem, where nontrivial Jordan blocks may arise, as the following counterexample for $A = I_4$ and $\hat{A} = A + XC^T$ shows:
$$ \hat{A} = \begin{bmatrix} 1 & & & \\ & 1 & & \\ & & 1 & \\ & & & 1 \end{bmatrix} + \begin{bmatrix} 1 & 0 \\0 & 1 \\ 0 & 0 \\ 0 & 0 \end{bmatrix} \begin{bmatrix} 1 & 0 & 0 & 0 \\ 0 & 0 & 1 & 0 \end{bmatrix} = \begin{bmatrix} 2 & & & \\ & 1 & 1 & \\ & & 1 & \\ & & & 1 \end{bmatrix}.$$
In this example, the Segre characteristic of $1 \in \sigma(A)$ is $((1,1,1,1))$ while it is $((2,1))$ for $\hat{A}$.

\begin{theorem}  \label{lem:eigval_lambda}
Let $S \in \MM_{2n}(\CC)$ be symplectic with $\lambda_1 \in \sigma(S)$ and let $\mu \in \CC \setminus \lbrace 0, \lambda_1, \lambda_1^{-1} \rbrace$ be given. Suppose that $\hat{S} = S + XRX^TJ^TS \in \MM_{2n}(\CC)$ has been constructed according to Theorem \ref{thm:symp_change1}. Then the following hold:
\begin{enumerate}
\item[$(i)$] If the Segre characteristic of $\lambda_1 \neq \pm 1$ as an eigenvalue of $S$ is
\begin{equation} ((s_k, s_{k-1}, \ldots , s_2, s_1)) \label{equ:segre} \end{equation}
with\footnote{Notice that for Theorem \ref{thm:symp_change1} to be applicable to $\lambda_1 \neq \pm 1$, $s_1=1$ is a necessary condition according to Theorem \ref{thm:nec_condition}. If $\lambda_1 = \pm 1$, then $s_1=1$ implies $s_2 = 1$ since then Jordan blocks of a particular size must appear an even number of times in the Jordan structure of $S$.} $s_k \geq s_{k-1} \geq \cdots \geq s_2 \geq s_1 = 1$, then the Segre characteristic of $\lambda_1$ as an eigenvalue of $\hat{S}$ is $((s_{k}, s_{k-1}, \ldots , s_2))$. Moreover, if $\lambda_1 = \pm 1$ and \eqref{equ:segre} is its Segre characteristic of $S$ with$^7$ $s_2 = s_1 = 1$, then the Segre characteristic of $\lambda_1$ as an eigenvalue of $\hat{S}$ is $((s_k, s_{k-1}, \ldots , s_3))$.
\item[$(ii)$] Let $\mu \notin \sigma(S)$. Then the Segre characteristic of $\mu$ as an eigenvalue of $\hat{S}$ is always $((1))$ if $\mu \neq \mu^{-1}$. If $\mu = \mu^{-1}$ its Segre characteristic is $((1,1))$ if and only if $R = R_1$ or $R= R_2$ from \eqref{equ:min_R}, otherwise it is $((2))$.
\end{enumerate}
\end{theorem}
\begin{proof}

$(i)$ We first assume $\lambda_1 \neq \lambda_1^{-1}$. According to the Segre characteristic $((s_k, \ldots , s_1))$ of $\lambda_1 \in \sigma(S)$ there are $k \geq 1$ Jordan blocks $L_k, \ldots , L_1$ of sizes $s_k, \ldots , s_1$. As $s_1 = 1$ let $x_1$ be the corresponding eigenvector. We denote the generalized eigenvectors corresponding to the $\ell$-th Jordan block $L_\ell$ by $x_1^{\ell}, \ldots , x_{s_\ell}^{\ell}$ and set $X_\ell = [\, x_1^{\ell} \; \cdots \; x_{s_\ell}^\ell \, ]$ and $\tilde{X} := [ \, X_2 \; \cdots \; X_k \, ].$ Since $\lambda_1^{-1} \in \sigma(S)$ has the same Segre characteristic as $\lambda_1$, there are also $k$ Jordan blocks $G_k, \ldots , G_1$ of $S$ for $\lambda_1^{-1}$ of sizes $s_k, \ldots , s_1$. Let $y_1$, $Y_\ell = [\, y_1^{\ell} \; \cdots \; y_{s_\ell}^\ell \, ]$ and $\tilde{Y} := [ \, Y_2 \; \cdots \; Y_k \, ]$ be defined analogously from the (generalized) eigenvectors for $\lambda_1^{-1}$. We now define the matrix $U := [ \, x_1 \; y_1 \; \tilde{X} \; \tilde{Y} \, ] \in \MM_{2n \times 2p}(\CC)$.
Then the matrix $\hat{S} = S + XRX^TJ^TS$ can be written as
$$ \hat{S} = S + \begin{bmatrix} x_1 & y_1 & \tilde{X} & \tilde{Y} \end{bmatrix} \begin{bmatrix} r_{11} & r_{12} \\ r_{21} & r_{22} \\ 0 & 0 \\ \vdots & \vdots \\ 0 & 0  \end{bmatrix} X^TJ^TS =: S + U \tilde{R} X^TJ^TS. 
$$
As $SU = U P'$ for $P' := \textnormal{diag}(\lambda_1, \lambda_1^{-1}, L_2, \ldots , L_k, G_2, \ldots , G_k)$ we therefore obtain $\hat{S}U = SU + U \tilde{R} X^TJ^TSU$ and thus
$\hat{S}U =U (P' + \tilde{R}X^TJ^TU P' )$.
Now set $P = [p_{ij}]_{ij} := P' + \tilde{R}X^TJ^TU P'$ and notice that the third to last row of $\tilde{R}X^TJ^TU P'$ are identically zero (due to the form of $\tilde{R}$). Therefore, the form of $P$ can be explicitly determined (with $\star$ indicating zero or nonzero entries that are not of further interest):
\begin{equation} P = \left[ \begin{array}{cc|ccc|ccc}
\lambda_1(1+r_{12}) & - \lambda_1^{-1}r_{11} & \star & \cdots & \star & \star & \cdots & \star \\
\lambda_1r_{22} & \lambda_1^{-1}(1-r_{21}) & \star & \cdots & \star & \star & \cdots & \star \\ \hline
0 & 0 & L_2 & & & & & \\
\vdots & \vdots & & \ddots & & & 0 & \\
\vdots & \vdots & & & L_k & & & \\ \hline
\vdots & \vdots & & & & G_2 & & \\
\vdots & \vdots & & 0 & &  & \ddots & \\
0 & 0& & & &  &  & G_k \\
 \end{array} \right]. \label{equ:P} \end{equation}

Now, its is easily seen that $L_2, \ldots , L_k, G_2, \ldots , G_k$ are part of the Jordan structure of $P$, hence they also arise in the Jordan structure of $\hat{S}$. This shows that the Segre characterstic of $\lambda_1 \neq \pm 1 \in \sigma(\hat{S})$ and $\lambda_1^{-1} \in \sigma(\hat{S})$ is both $((s_k, s_{k-1}, \ldots , s_2))$. If $\lambda = \pm 1$ the proof follows the same lines without the use of $\tilde{Y}$.
To prove $(ii)$ we note that the upper-left $2 \times 2$ block of $P$ is exactly $\Omega$ from \eqref{equ:omega1}, hence its eigenvalues are $\mu, \mu^{-1}$. If $\mu \neq \mu^{-1}$, then $\Omega$ is semisimple. Therefore, we obtain the Segre characteristic of $\mu$ and $\mu^{-1}$ as eigenvalues of $\hat{S}$ both as $((1))$. On the other hand, if $\mu = \mu^{-1}$, then $\Omega$ is semisimple if and only if its minimal polynomial is $p(z) = z - \mu$. Now
$$p(\Omega') = \begin{bmatrix} \lambda_1(1+r_{12}) - \mu & - \lambda_1^{-1}r_{11} \\ \lambda_1 r_{22} & \lambda_1^{-1}(1-r_{21}) - \mu \end{bmatrix}.$$
Thus, for $p(\Omega)=0$ we must have $r_{11}=r_{22}=0$ and the only possible choices for $\Omega$ to be semisimple are $R_1$ and $R_2$ from \eqref{equ:min_R}. It is easy to check that in fact both choices result in $p(\Omega)=0$. This proves that the Segre characteristic of $\mu$ as an eigenvalue of $\hat{S}$ is $((1,1))$ if $R=R_1,R_2$ and that it has to be $((2))$ otherwise.
\end{proof}

Finally, assume that $\mu$ was already an eigenvalue of $S$ and Theorem \ref{thm:symp_change1} is applied for $\lambda_1 \in \sigma(S)$. Then the arguments of the proof of Lemma \ref{lem:eigenvectors_0} apply to $\mu$ (as an 'old'  eigenvalue of $S$) as well as the result from Theorem \ref{lem:eigval_lambda} $(ii)$ (for $\mu$ as the 'new' eigenvalue appearing in the spectrum of $\hat{S}$). Thus, the Segre characteristic of $\mu \in \sigma(\hat{S})$ is its old Segre characteristic from $S$ extended by one of the cases described in Theorem \ref{lem:eigval_lambda} $(ii)$.

As another consequence of Theorem \ref{lem:eigval_lambda} it follows that, if $\lambda_1$ is semisimple for $S$, then $\lambda_1$ remains semisimple for $\hat{S}$ is case its multiplicity was $\geq 2$.
Moreover, assume the symplectic matrix $S \in \MM_{2n}(\CC)$ is diagonalizable (i.e. all its eigenvalues are semisimple) and let $\hat{S}$  be constructed according to Theorem \ref{thm:symp_change1}. As a consequence of Lemma \ref{lem:eigenvectors_0} and Theorem \ref{lem:eigval_lambda}, the diagonalizability of $\hat{S}$ can then only be circumvented in case a $2 \times 2$ Jordan block arises for $\mu = \mu^{-1}$. As seen above, a $2 \times 2$ Jordan block for $\mu$ will arise if $R$ is different from $R_1$ and $R_2$. In other words, if $R_1$ or $R_2$ from \eqref{equ:min_R} are chosen in Theorem \ref{thm:symp_change1}, the matrix $\hat{S}$ will be semisimple in case $S$ was. In fact, this is the only situation in which $S$ and $\hat{S}$ are simultaneously diagonalizable since the simultaneous diagonalizability of $S$ and $\hat{S}$ is only possible if $S$ and $\hat{S}$ commute. According to Theorem \ref{thm:commute} this is the case if and only if $R=R_1,R_2$ are chosen. We summarize this result in the following corollary.

\begin{corollary}
Let $S \in \MM_{2n}(\CC)$ be symplectic with $\lambda_1 \in \sigma(S)$ and let $\mu \in \CC \setminus \lbrace 0, \lambda_1, \lambda_1^{-1} \rbrace$ be given. Suppose that $\hat{S} = S + XRX^TJ^TS \in \MM_{2n}(\CC)$ has been constructed according to Theorem \ref{thm:symp_change1} and that $S$ is semisimple. Then $\hat{S}$ is semisimple if and only if $R=R_1$ or $R=R_2$ for one of the matrices in \eqref{equ:min_R}. Moreover, in this case, $S$ and $\hat{S}$ are simultaneously diagonalizable.
\end{corollary}

Example \ref{ex:sim_diag} below shows how the simultaneous diagonalization looks like if $R_1$ or $R_2$ are used to construct $\hat{S}$.

\begin{example} \label{ex:sim_diag}
Whenever $T^{-1}ST = \textnormal{diag}(\Lambda, \Lambda^{-1})$ with $T = [ \; x_1 \; Y \; x_2 \; Z \, ] \in \MM_{2n}(\CC)$ and $Y, Z \in \MM_{2n \times (n-1)}(\CC)$ and $\Lambda = \textnormal{diag}(\lambda_1, \ldots , \lambda_n)$, then for $S_j = S + XR_jX^TJ^TS$ with $R_1,R_2$ from \eqref{equ:min_R} we have
$$ T^{-1}\hat{S}_1 T = \begin{bmatrix} \Lambda_1 & \\ & \Lambda_1^{-1} \end{bmatrix}, \quad \textnormal{and} \quad T^{-1}\hat{S}_2 T = \begin{bmatrix} \Lambda_2 & \\ & \Lambda_2^{-1} \end{bmatrix}$$
where $\Lambda_1 = \textnormal{diag}(\mu, \lambda_2, \ldots, \lambda_n)$ and $\Lambda_2 = \textnormal{diag}(\mu^{-1}, \lambda_2, \ldots , \lambda_n)$. In fact, for $S$ being diagonal, $R_1$ and $R_2$ are the only possible choice such that $\hat{S}$ is diagonal, too.
\end{example}

\subsection{A note on condition numbers}

We conclude this section with a result on eigenvalue condition numbers. It is a surprising fact, that we can apply Theorem \ref{thm:symp_change1} to a symplectic matrix $S \in \MM_{2n}(\CC)$ without changing any eigenvalue condition number (for simple eigenvalues) at all. For Rado's theorem this is not true: an unstructured application of Theorem \ref{thm:rado} typically changes all eigenvalue condition numbers, even those of eigenvalues that remain unchanged. The main result on the behavior of condition numbers is stated in Theorem \ref{thm:condition}. For its proof, we need the following well-known fact that we state without proof in Lemma \ref{lem:leftright}.

\begin{lemma} \label{lem:leftright}
Let $A \in \MM_n(\CC)$ and suppose $x$ is a right eigenvector of $A$ for $\lambda$ (i.e. $Ax = \lambda x$) and $y$ is a left eigenvector of $A$ for $\mu$ (i.e. $y^HA = \mu y^H$). Then $y^Hx = 0$ if $\lambda \neq \mu$.
\end{lemma}

\begin{theorem} \label{thm:condition}
Let $S \in \MM_{2n}(\CC)$ be symplectic with eigenvalues $\lambda_1, \lambda_1^{-1}, \lambda_2, \lambda_2^{-1},$ $\ldots , \lambda_n, \lambda_n^{-1}$ and let $\mu \in \CC \setminus \lbrace 0 \rbrace$ with $\mu \notin \sigma(S)$ be given. Assume that $\lambda_1$ is simple and let 
$ \hat{S} = S + XRX^TJ^TS$
be constructed according to Theorem \ref{thm:symp_change1} so that $\mu, \mu^{-1}, \lambda_2, \lambda_2^{-1}, \ldots , \lambda_n, \lambda_n^{-1}$ are the eigenvalues of $\hat{S}$. Then the following holds:
\begin{enumerate}
\item[$(i)$] If $\nu$ is a simple eigenvalue of $\hat{S}$, $\mu \neq \nu \neq \mu^{-1}$, then $\kappa(\hat{S}, \nu) = \kappa(S, \nu).$
\item[$(ii)$] For $R = R_1$, where $R_1$ is the matrix given in \eqref{equ:min_R}, it holds that
$$\kappa(\hat{S}, \mu) = \kappa(S, \lambda_1) \quad \textnormal{and} \quad \kappa(\hat{S}, \mu^{- 1}) = \kappa(S, \lambda_1^{- 1}).$$
\end{enumerate}

\end{theorem}

\begin{proof}
$(i)$ Under the assumption $\mu^{-1} \neq \nu \neq \mu$ and the simplicity of $\lambda_1$, it follows that $\nu = \lambda_j$ or $\nu = \lambda_j^{-1}$ for some $j=2, \ldots , n$. Assume w.\,l.\,o.\,g. that $\nu = \lambda_2$ and let $y_1 \in \CC^{2n}$ be some corresponding eigenvector of $S$. According to Lemma \ref{lem:eigenvectors_0}, $\hat{S}y_1 = \lambda_2y_1$ holds. Next suppose $z_1 \in \CC^{2n}$ is some left eigenvector of $S$ for $\lambda_2$, i.e. $z_1^HS = \lambda_2 z_1^H$. Then
$$ z_1^H \hat{S} = z_1^H \big( S + XRX^TJ^TS \big) = z_1^HS = \lambda_2 z_1^H$$
as $z_1^HX = 0$ according to Lemma \ref{lem:leftright}. Thus the left and right eigenvectors for $\lambda_2$ of $S$ and $\hat{S}$ coincide, which directly implies $ \kappa(\hat{S}, \lambda_2) = \kappa(S, \lambda_2)$.

$(ii)$ Now assume $R=R_1$ for $R_1$ in \eqref{equ:min_R} and suppose $\nu = \mu$. The condition $\mu \notin \sigma(S)$ implies $\mu$ to be simple for $\hat{S}$. If $x_1,x_2 \in \CC^{2n}$ are eigenvectors of $S$ for $\lambda_1$ and $\lambda_1^{-1}$, respectively, then one directly obtains
$$ \begin{aligned} \hat{S}x_1 &= Sx_1 + XRX^TJ^TSx_1 = \lambda_1x_1 - \lambda_1 XRX^TJx_1  = \lambda_1x_1 - \lambda_1 XR \begin{bmatrix} 0 \\ -1 \end{bmatrix} \\
&= \lambda_1(1+r_{12})x_1 = \lambda_1 \cdot (1 + \lambda_1^{-1}(\mu - \lambda_1))x_1 = (\lambda_1 + \mu - \lambda_1)x_1 = \mu x_1
\end{aligned} $$
(similarly, $\hat{S}x_2 = \mu^{-1}x_2$ follows).
Analogously we have $x_2^TJ^T \hat{S} = \mu x_2^TJ^T$, which follows from $x_2^TJ^TS = \lambda_1 x_2^TJ^T$ (see \eqref{equ:rightleft}). Thus, the left and right eigenvectors of $S$ for $\lambda_1$ and those of $\hat{S}$ for $\mu$ coincide and the statement follows. The proof is analogous for $\nu = \mu^{-1}$.
\end{proof}

\begin{remark}
One can proceed as in the above proof to see that, if $R=R_2$ is used,
$$\kappa(\hat{S}, \mu) = \kappa(S, \lambda_1^{-1}) \quad \textnormal{and} \quad \kappa(\hat{S}, \mu^{- 1}) = \kappa(S, \lambda_1).$$
In fact, it is easy to show that now $\hat{S}x_1 = \mu^{-1}x_1$ and $\hat{S}x_2 = \mu x_2$ hold.
\end{remark}

\section{Experiments}
\label{sec:experiments}

To perform numerical experiments in \textsc{Matlab} R2021b, we used the code available in \cite{jagger}. To obtain symplectic matrices $S \in \MM_{200}(\CC)$, a symplectic matrix $S' \in \MM_{200}(\CC)$ constructed from \cite{jagger}, was modified as
\begin{equation} S = \begin{bmatrix} D_1^{-1} & \\ & D_1 \end{bmatrix} S' \begin{bmatrix} D_2 & \\ & D_2^{-1} \end{bmatrix}, \quad \textnormal{where} \quad \begin{aligned} D_1 &= \textnormal{diag}(\alpha_1, \ldots , \alpha_{100} ), \\ D_2 &= \textnormal{diag}(\beta_1, \ldots , \beta_{100}) \end{aligned} \label{equ:rand_construct} \end{equation}
and $\alpha_i, \beta_j$ are random numbers: \texttt{rand(1) + 1i*rand(1)}. Compared to $S'$, $S$ has a more widespread spectrum ranging in magnitude from about $10^{-3}$ to $10^2$. 

  \begin{figure}
    \begin{center}
\includegraphics[width=150pt]{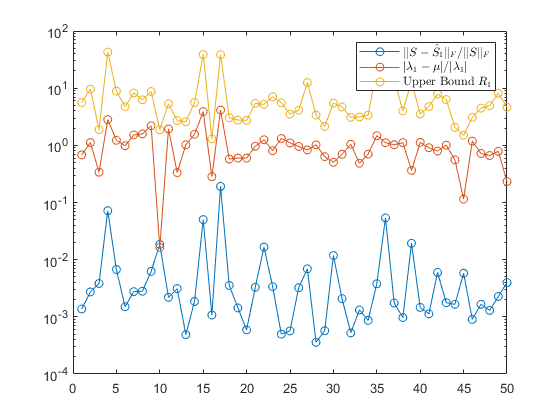} \includegraphics[width=150pt]{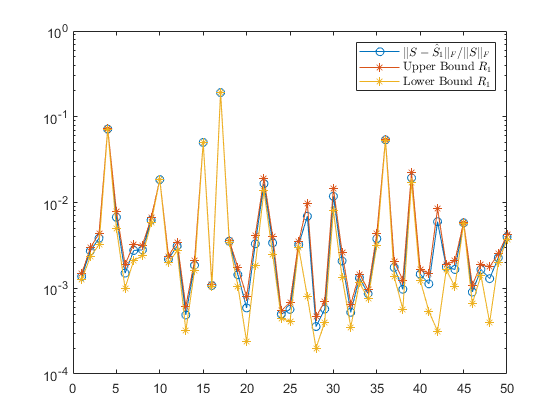}
 \end{center}
        \caption{The plots show the effect of an eigenvalues modification when $\lambda_1 \in \sigma(S)$  undergoes a relative change of $|\lambda_1|$. Fifty experiments have been performed with $S \in \MM_{200}(\CC)$. The plots show the use of $\hat{S}_1 = (I_{2n} + XR_1X^TJ^T)S$ for $R_1$ in \eqref{equ:min_R}, the coarse bound in \eqref{equ:frob_bound_S1a} (left plot), the upper and lower bounds from Theorem \ref{thm:coarse_bound} (right plot) and the relative change in the eigenvalue $| \lambda_1 - \mu|/|\lambda_1|$.}
    \label{fig:upper}
    \end{figure}

In Figure \ref{fig:upper}, 50 examples are shown where a randomly selected eigenvalue $\lambda_1$ of $S$ has been changed to $\mu = \lambda_1(1 + \gamma z)$, where $z \in \CC$ is a random complex number with $| z| = 1$ and $\gamma = | \lambda_1|$. Therefore, $|\mu - \lambda_1|/|\lambda_1| = | \lambda_1|$. The plot shows the relative change $\Vert S - \hat{S}_1 \Vert_F/\Vert S \Vert_F$, the coarse bound in \eqref{equ:frob_bound_S1a}, the upper and lower bounds from Theorem \ref{thm:coarse_bound} and the relative change in the eigenvalue $| \lambda_1 - \mu|/|\lambda_1|$. The plots show clearly that the bounds from Theorem \ref{thm:coarse_bound} are significantly sharper than the bound from \eqref{equ:frob_bound_S1a}. If $R_2$ is used instead of $R_1$ the plots do not alter essentially.

The most significant difference between using $\hat{S}_1$ and $\hat{S}_2$ can be seen when an eigenvalue is subject to a small or large relative change compared to its absolute value. This is shown in Figure \ref{fig:R1R2} where the experimental set-up is the same as before but now $\gamma = 10^{-3}|\lambda_1|$ (left plot) and $\gamma = 10^3|\lambda_1|$ (right plot) are chosen. Both plots show the relative change $\Vert S - \hat{S}_j \Vert_F/\Vert S \Vert_F$ for $j=1,2$ and the relative change in the eigenvalue $| \lambda_1 - \mu|/|\lambda_1|$. It is seen that for a small change in the eigenvalue $\lambda_1$, the relative change $\Vert S - \hat{S}_1 \Vert_F/\Vert S \Vert_F$ is significantly smaller than $\Vert S - \hat{S}_2 \Vert_F/\Vert S \Vert_F$. However, when $\lambda_1$ undergoes a large change, then there is no big difference in using $R_1$ or $R_2$.

In Figure \ref{fig:surface} we show the use of matrices $R$ different from $R_1$ and $R_2$ in \eqref{equ:min_R}. To this end, we fix a symplectic matrix $S \in \MM_{200}(\CC)$ and an eigenvalue $\lambda_1$ of $S$ that is to be modified by a relativ change of $| \lambda_1|$.  We consider a mesh grid on $[-1,1] \times [-1,1]$ with 150 discretization points in each direction. Each point $(x_j,y_k)$ is associated with the complex number $c_{jk} := x_j + \imath y_k$ from which we made up $r_{11} = r_{22} = \sqrt{c_{jk}}$. The values for $r_{12}$ and $r_{21}$ are found according to \eqref{equ:r12r21}  and \eqref{equ:eta_pol} so that we obtain two different matrices $\tilde{R}_1 = R(\eta_1,r_{11},r_{22})$ and $\tilde{R}_2 = R(\eta_2,r_{11},r_{22})$. The matrices $\tilde{S}_j = (I_{2n} + X\tilde{R}_jX^TJ^T)S$, $j=1,2$, where constructed and in Figure \ref{fig:surface} the minimum of $\Vert S - \tilde{S}_j \Vert_F/\Vert S \Vert_F$ is shown for each point $c_{jk} \in [-1,1] \times [-1,1]$. The plot indicates that the minimum numerically found among all values $\Vert S - \tilde{S}_j \Vert_F/\Vert S \Vert_F$ is attained for $c_{jk} = 0$, i.e. $r_{11} = r_{22} = 0$. Thus, the minimum is obtained for one of the matrices $R$ in \eqref{equ:min_R}. In most examples that have been considered a plot similar to the one in Figure \ref{fig:surface} arose. However, seldom the numerical minimum was detected somewhere near $c_{jk} = 0$.

  \begin{figure}
    \begin{center}
\includegraphics[width=150pt]{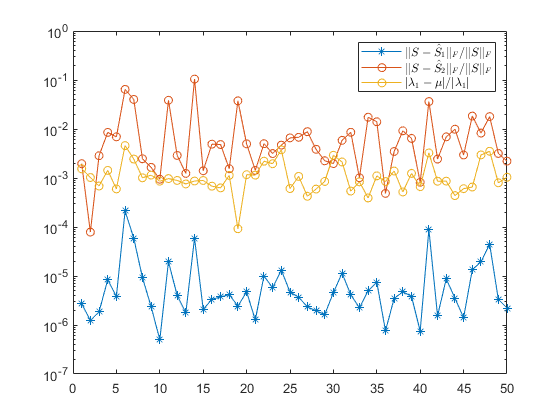} \includegraphics[width=150pt]{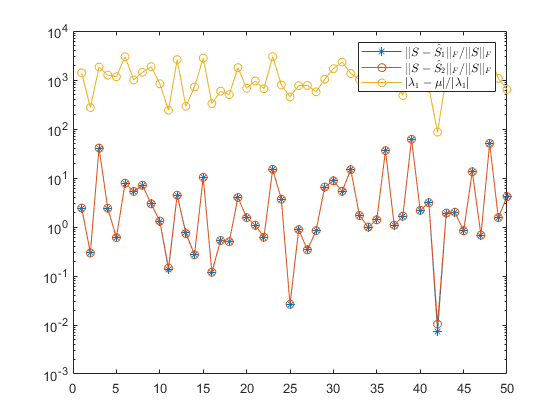}
 \end{center}
        \caption{The plots show the effect of an eigenvalues modification when $\lambda_1 \in \sigma(S)$  undergoes a relative change of $0.001 \cdot |\lambda_1|$ (left plot) and $1000 \cdot | \lambda_1|$ (right plot). Fifty experiments have been performed with $S \in \MM_{200}(\CC)$. The plots show the use of $\hat{S}_1$ and $\hat{S}_2$ and the relative change in the eigenvalue $| \lambda_1 - \mu|/|\lambda_1|$.}
    \label{fig:R1R2}
    \end{figure}

  \begin{figure}
    \begin{center}
\includegraphics[width=170pt]{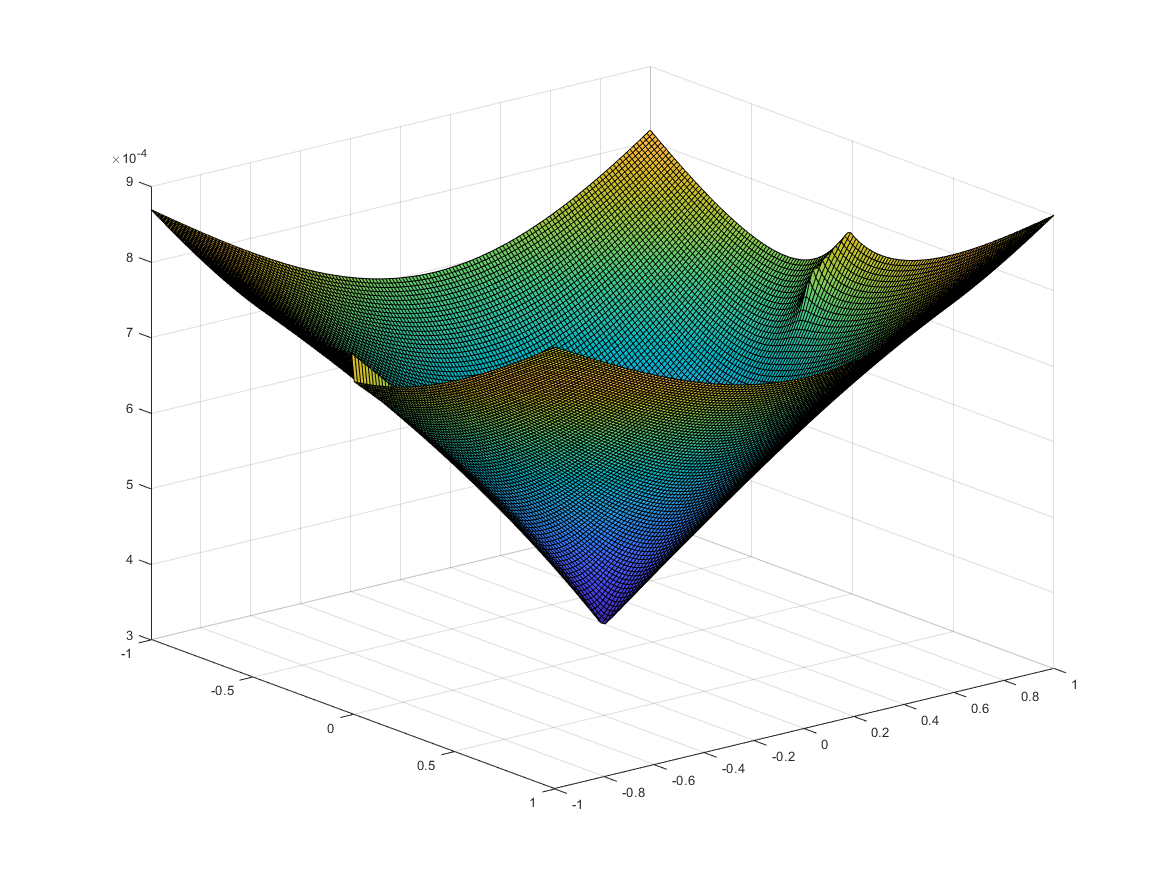}
 \end{center}
        \caption{The plot shows the minimum of $\Vert S - \tilde{S}_j \Vert_F/\Vert S \Vert_F$ for $j=1,2$ on the square $[-1,1] \times [-1,1]$ with 150 discretization points in each direction. The minimum on the grid is attained for $c_{jk}=0$, i.e. $r_{11} = r_{22}=0$. For random experiments the plots always look similar to the plot shown above. Notice that the surface has a sharp edge which arises due to the complex square root that has to be computed.}
    \label{fig:surface}
    \end{figure}

\section{Symplectic Matrix Pencils}
\label{sec:pencils}

In this section we analyse the eigenvalue modification problem of Section \ref{sec:intro} for symplectic matrix pencils.
We call a matrix pencil $P(\lambda) = A - \lambda B$, $A,B \in \MM_{2n}(\CC)$ symplectic (see \cite[Sec.\,2.1.1]{fass}) if it holds that
\begin{equation} AJA^T = BJB^T. \label{equ:symp_pencil} \end{equation}
From \eqref{equ:symp_pencil} it follows that for a symplectic pencil $P(\lambda) = A - \lambda B$ either $A$ and $B$ are both regular or singular, cf. \cite[Sec.\,2]{fass}. A scalar $\nu \in \CC$ is called an eigenvalue of $P(\lambda)$ if there exists some nonzero $x \in \CC^{2n}$ with $P(\nu)x=0$, i.e. $Ax = \nu B x$ \cite[Sec.\,2]{betcke}. Eigenvalues of symplectic pencils also arise in pairs $(\nu, \nu^{-1})$ \cite{fass}. In particular, if $A$ and $B$ are singular, it is also possible to have $\nu=0$ as an eigenvalue of $P(\lambda)$ which implies that $\lambda^{-1} = \infty$ is also an eigenvalue of $P(\lambda)$ (see, e.g., \cite{deteran} for more information on the finite and infinite eigenstructure of matrix pencils). In the following, we focus on symplectic pencils $P(\lambda) =A  - \lambda B$ where both $A$ and $B$ are regular. Thus, $P(\lambda)$ has neither the eigenvalue zero nor the eigenvalue infinity.

In general, for an arbitrary matrix pencil $P(\lambda) = A - \lambda B$, $A, B \in \MM_n(\CC)$, where $B$ is nonsingular, we can easily derive an adapted version of Rado's theorem.

\begin{theorem}
Let $P(\lambda) = A - \lambda B$, $A,B \in \MM_{n}(\CC)$ and $B$ nonsingular, be a matrix pencil with eigenvalues $\lambda_1,  \ldots , \lambda_n \in \CC $. Let $x_1, \ldots , x_k \in \CC^n$ be eigenvectors for $\lambda_1, \ldots , \lambda_k$ such that $\textnormal{rank}(X)=k$ for $X = [\, x_1 \; x_2 \; \cdots \; x_k \, ] \in \MM_{n \times k}(\CC)$. Furthermore, let $C \in \MM_{n \times k}(\CC)$ be arbitrary. Then the matrix pencil
$$ \tilde{P}(\lambda) = (A + BXC^T) - \lambda B$$ has the eigenvalues $\mu_1, \ldots , \mu_k, \lambda_{k+1}, \ldots , \lambda_n$, where $\mu_1, \ldots , \mu_k$ are the eigenvalues of the matrix $\Lambda + C^TX$ with $\Lambda = \textnormal{diag}(\lambda_1, \ldots , \lambda_k)$.
\end{theorem}

\begin{proof}
As $B^{-1}P(\lambda) = B^{-1}A - \lambda I_n$, the eigenvalues of $P(\lambda)$ coincide with those of the matrix $M := B^{-1}A \in \MM_{2n}(\CC)$. In particular, $P(\lambda)$ does not have $\lambda = \infty$ as an eigenvalue.
Assume that $AX = BX \Lambda$ with $\Lambda = \textnormal{diag}(\lambda_1, \ldots , \lambda_k)$ we obtain $B^{-1}AX = X \Lambda$ and Theorem \ref{thm:rado} implies that
$\hat{M} := B^{-1}A + XC^T$ has the eigenvalues $\mu_1, \ldots , \mu_k, \lambda_{k+1}, \ldots , \lambda_n$, where $\mu_1, \ldots , \mu_k$ are the eigenvalues of the matrix $\Lambda + C^TX$. As $\hat{M}$ and the matrix pencil $\hat{P}(\lambda) = (A + BXC^T) - \lambda B$ have the same eigenvalues, the statement follows.
\end{proof}

Now let $P(\lambda) = A - \lambda B$ be a symplectic matrix pencil according to \eqref{equ:symp_pencil} with nonsingular $A,B$ and eigenvalues $\lambda_1, \lambda_1^{-1}, \ldots , \lambda_n, \lambda_n^{-1}$. We set $$\hat{P}(\lambda) := (A + BXC^T) - \lambda B$$ and intend to determine $C \in \MM_{2n \times 2}(\CC)$ such that $\hat{P}(\lambda)$ is again symplectic and has the eigenvalues $\mu, \mu^{-1}, \lambda_2, \lambda_2^{-1}, \ldots , \lambda_n, \lambda_n^{-1}$ for a given value $\mu \in \CC \setminus \lbrace 0 \rbrace$.
A direct calculation reveals that \eqref{equ:symp_pencil} is equivalent to $B^{-1}A$ being a symplectic matrix, i.e. $J^T(B^{-1}A)^TJ = (B^{-1}A)^{-1}$. Thus, Theorem \ref{thm:symp_change1} can be applied to the matrix $S := B^{-1}A$ that has the same eigenvalues as $P(\lambda)$. Now suppose
$$ AX = BX \begin{bmatrix} \lambda_1 & \\ & \lambda_1^{-1} \end{bmatrix}, \qquad X = \begin{bmatrix} x_1 & x_2 \end{bmatrix} \in \MM_{2n \times 2}(\CC),$$
i.e. $x_1,x_2 \in \CC^{2n}$ are generalized eigenvectors for $\lambda_1$ and $\lambda_1^{-1}$, respectively. Furthermore, assume $x_1^TJx_2 \neq 0$. Then
$$ \hat{S} := B^{-1}A + XRX^TJ^TB^{-1}A$$
is symplectic with eigenvalues $\mu, \mu^{-1}, \lambda_2, \lambda_2^{-1}, \ldots , \lambda_n, \lambda_n^{-1}$ provided that $R$ is chosen according to the conditions in Theorem \ref{thm:symp_change1} for $\mu$. Then, the matrix pencil $\hat{P}(\lambda) := B(\hat{S} - \lambda I_{2n})$, i.e.
\begin{equation} \hat{P}(\lambda) = \big( A + BXRX^TJ^TB^{-1}A \big) - \lambda B, \label{equ:hat_P} \end{equation}
has the same eigenvalues as $\hat{S}$ \cite[Sec.\,3.1]{deteran}. In fact, $\hat{P}(\lambda)$ is again a symplectic pencil. To show this, we have to check that
$$ \big( A+BXRX^TJ^TB^{-1}A \big) J \big( A+BXRX^TJ^TB^{-1}A \big)^T = BJB^T$$
holds. To this end, it only remains to prove that
\begin{equation} \begin{aligned} &AJ \big( BXRX^TJ^TB^{-1}A \big)^T + \big( BXRX^TJ^TB^{-1}A \big)JA^T \\ & \qquad + \big( BXRX^TJ^TB^{-1}A \big) J \big( BXRX^TJ^TB^{-1}A \big)^T = 0
\end{aligned} \label{equ:symp_pencil_change1} \end{equation}
since $AJA^T = BJB^T$ holds by assumption. Using this relation, \eqref{equ:symp_pencil_change1} simplifies to
$$ -BXR^TX^TB^T + BXRX^TB^T + BXRJ_2R^TX^TB^T = 0$$
which can be rewritten as
$BX \big( R - R^T + RJ_2R^T \big) X^TB^T = 0$. As in Section \ref{sec:symplectic_matrix} this relation holds if and only if $R - R^T + RJ_2R^T =0$.
As for any $R \in \MM_2(\CC)$ we have $RJ_2R^T = R^TJ_2R$, it follows that the condition $R-R^T+RJ_2R^T=0$ is equivalent to \eqref{equ:symp_matrixequation2}. Therefore, since $R$ was constructed according to \eqref{equ:symp_change1b} so that $R - R^T + R^TJ_2R = 0$ holds, this finally shows that \eqref{equ:symp_pencil_change1} is true, so $\hat{P}(\lambda)$ is symplectic. As already mentioned above, $\hat{P}(\lambda)$ and $\hat{S}$ have the same eigenvalues, so the eigenvalues of $\hat{P}(\lambda)$ are $\mu, \mu^{-1}, \lambda_2, \lambda_2^{-1}, \ldots , \lambda_n, \lambda_n^{-1}$.

Notice that  $\hat{P}(\lambda)$ in \eqref{equ:hat_P} can be rewritten in various ways, e.g.
\begin{align} \hat{P}(\lambda) &= \big( A + BXRX^TB^TA^{-T}J^T \big)- \lambda B \label{equ:symp_form1} \\ &= \big( A + BXRX^TB^TJ^TA \big) - \lambda B \label{equ:symp_form2} \end{align}
using the relation $J^TB^{-1}A = B^TA^{-T}J^T$ that follows from $B^{-1}A$ being symplectic in \eqref{equ:symp_form1} and $A^{-T} = J^TAJ$ in \eqref{equ:symp_form2}. Using  $X^TB^T = \Lambda^{-1}X^TA^T$  and exchanging $X^TB^T$ with $\Lambda^{-1}X^TA^T$ in \eqref{equ:symp_form1} yields
\begin{align} \hat{P}(\lambda) &= (A + BXR \Lambda^{-1}X^TA^TA^{-T}J^T)- \lambda B \notag \\ &= \left( A + BXR \begin{bmatrix} \lambda_1^{-1} & \\ & \lambda_1 \end{bmatrix} X^TJ^T \right) - \lambda B. \label{equ:pencil_rewritten} \end{align}
which is an expression for $\hat{P}(\lambda)$ similar to \eqref{equ:hat_S}. We conclude our finding in the following theorem.

\begin{theorem} \label{thm:symp_pencil_change}
Let $P(\lambda) = A - \lambda B$, $A,B \in \MM_{2n}(\CC)$ nonsingular, be a symplectic pencil with eigenvalues $\lambda_1, \lambda_1^{-1}, \lambda_2, \lambda_2^{-1},$ $\ldots , \lambda_n, \lambda_n^{-1} \in \CC$ and let $\mu \in \CC \setminus \lbrace 0 \rbrace$ be given. Let $x_1, x_2 \in \CC^{2n}$ be eigenvectors for $\lambda_1$ and $\lambda_1^{-1}$, respectively, normalized such that $X^TJ_{2n}X = J_2$ for $X = [ \; x_1 \; x_2 \; ] \in \MM_{2n}(\CC)$ and set $d := (\mu + \mu^{-1} ) - ( \lambda_1 + \lambda_1^{-1})$. Then the matrix pencil
\begin{equation} \hat{P}(\lambda) := \left( A + BXR \begin{bmatrix} \lambda_1^{-1} & \\ & \lambda_1 \end{bmatrix}X^TJ \right) - \lambda B \label{equ:symp_pencil_change} \end{equation}
is again symplectic and has the eigenvalues $\mu, \mu^{-1}, \lambda_2, \lambda_2^{-1}, \ldots , \lambda_n, \lambda_n^{-1}$ provided that $R=[r_{ij}]_{ij} \in \MM_2(\CC)$ is chosen such that \eqref{equ:symp_change1a} and \eqref{equ:symp_change1b} hold.
\end{theorem}
Regarding $\hat{P}(\lambda)$ in \eqref{equ:hat_P}, let $\hat{P}(\lambda) = \hat{A} - \lambda \hat{B}$ with $\hat{A} = A + BXRX^TJ^TB^{-1}A$ and $\hat{B} = B$. Then certainly $\Vert \hat{B} - B \Vert / \Vert B \Vert = 0$ while
$$ \frac{\Vert \hat{A} - A \Vert}{\Vert A \Vert} \leq \kappa(B) \Vert R \Vert \Vert X \Vert^2,$$
where $\kappa(B) = \Vert B \Vert \Vert B^{-1} \Vert$ is the condition number of the matrix $B$. Thus, using Theorem \ref{thm:coarse_bound} we may immediately bound $\Vert \hat{A} -A \Vert/ \Vert A \Vert$. Furthermore, from the form of $\hat{P}(\lambda)$ in \eqref{equ:symp_pencil_change}, statements similar to those in Section \ref{sec:eigenvectors} can directly be derived.

\section{Summary}
In this work we showed how to modify a pair of eigenvalues $\lambda, 1/\lambda$ of a symplectic matrix $S$ to desired target values $\mu, 1/\mu$ for a symplecitc matrix $\hat{S}$ in a structure-preserving way. Universal
bounds on the relative distance between $S$ and $\hat{S}$ with modified spectrum were given. The eigenvalues
Segre characteristics of $S$ were related to those of $S$ and some statements on eigenvalue condition numbers have been derived. The main results have been extended to matrix pencils.

\section{Acknowledgement}
The author is grateful to Thomas Richter as this work was in parts developed during the authors employment in Thomas Richter's group at the Otto-von-Guericke-Universit{\"a}t Magdeburg.

\end{document}